\documentclass[10pt,reqno]{amsart}
\usepackage[shortlabels]{enumitem}
\usepackage[style=numeric,sorting=nyt]{biblatex}
\usepackage{url}
\makeatletter

\renewcommand{\tocsection}[3]{%
 \indentlabel{\@ifnotempty{#2}{\bfseries\ignorespaces#1 #2\quad}}\bfseries#3}

\renewcommand{\tocsubsection}[3]{%
  \indentlabel{\@ifnotempty{#2}{\ignorespaces#1 #2\quad}}#3}

\newcommand\@dotsep{4.5}
\def\@tocline#1#2#3#4#5#6#7{\relax
  \ifnum #1>\c@tocdepth 
  \else
    \par \addpenalty\@secpenalty\addvspace{#2}%
    \begingroup \hyphenpenalty\@M
    \@ifempty{#4}{%
      \@tempdima\csname r@tocindent\number#1\endcsname\relax
    }{%
      \@tempdima#4\relax
    }%
    \parindent\z@ \leftskip#3\relax \advance\leftskip\@tempdima\relax
    \rightskip\@pnumwidth plus1em \parfillskip-\@pnumwidth
    #5\leavevmode\hskip-\@tempdima{#6}\nobreak
    \leaders\hbox{$\m@th\mkern \@dotsep mu\hbox{.}\mkern \@dotsep mu$}\hfill
    \nobreak
    \hbox to\@pnumwidth{\@tocpagenum{\ifnum#1=1\bfseries\fi#7}}\par
    \nobreak
    \endgroup
  \fi}
\AtBeginDocument{%
\expandafter\renewcommand\csname r@tocindent0\endcsname{0pt}
}
\def\l@subsection{\@tocline{2}{0pt}{2.5pc}{5pc}{}}
\makeatother

\everymath{\displaystyle}
\newcommand{\olsi}[1]{\,\overline{\!{#1}}} 
\newcommand{\ols}[1]{\mskip.5\thinmuskip\overline{\mskip-.5\thinmuskip {#1} \mskip-.5\thinmuskip}\mskip.5\thinmuskip} 
\newcommand{\ii}{\ensuremath{\sqrt{-1}}}
\newcommand{\pp}{\bar\partial}

 \usepackage{amsmath,amsfonts,amssymb,amsthm,mathrsfs}
 \usepackage{hyperref}
 \usepackage{bm}
 \usepackage{color}
 \usepackage{esint}
\usepackage{graphicx}
\usepackage{graphics}
\usepackage{geometry}
\usepackage[all]{xy}

\usepackage{tikz}

\usepackage{hyperref}

 \newtheorem{theorem}{Theorem}[section]

 \newtheorem{lemma}[theorem]{Lemma}
 \newtheorem{corollary}[theorem]{Corollary}
 \newtheorem{proposition}[theorem]{Proposition}
 \newtheorem{conjecture}[theorem]{Conjecture}
 \newtheorem{remark}[theorem]{Remark}
 \newtheorem*{claim}{Claim}
 \theoremstyle{definition}
 
 \newtheorem{question}[theorem]{Question}
 
  \newtheorem{problem}[theorem]{Problem}

\numberwithin{equation}{section}

\newcommand{\p}{\partial}
\newcommand{\bp}{\bar\partial}

\newcommand{\C}{\mathbb{C}}

\newcommand{\R}{\mathbb{R}}

\DeclareMathOperator{\diam}{diam}

\DeclareMathOperator{\Id}{Id}

\DeclareMathOperator{\Vol}{Vol}

\begin{document}
 \title{No semistability at infinity for  Calabi-Yau metrics asymptotic to cones}
 
 \author{Song Sun}
\address{Department of Mathematics, University of California, Berkeley, CA 94720 } 
\email{sosun@berkeley.edu}\author{Junsheng Zhang}
\address{Department of Mathematics, University of California, Berkeley, CA 94720 } 
\email{jszhang@berkeley.edu}

 \begin{abstract}
 We discover a ``no semistability at infinity" phenomenon for complete Calabi-Yau metrics asymptotic to cones, which is proved by eliminating the possible appearance of an intermediate K-semistable cone in the 2-step degeneration theory developed by Donaldson and the first author. It is in sharp contrast to the  setting of local singularities of K\"ahler-Einstein metrics. A byproduct of the proof is a polynomial convergence rate  to the asymptotic cone for such manifolds, which bridges the gap between the general theory of Colding-Minicozzi and the classification results of  Conlon-Hein.
 \end{abstract}

 \maketitle
 \section{Introduction}

Throughout this paper we denote by $(X, p, g, J, \omega, \Omega)$ a pointed complete Calabi-Yau manifold of complex dimension $n$. This means that $p\in X$,  $(g, J, \omega)$ is a complete K\"ahler metric, $\Omega$ is a holomorphic volume form and the following complex Monge-Amp\`ere equation is satisfied
\begin{equation*}\omega^n=(\sqrt{-1})^{n^2} \Omega\wedge\overline\Omega.	
\end{equation*}
 We will always impose the following two conditions
 \begin{itemize}
 		\item   Euclidean volume growth:  there exists a $\kappa>0$ such that  $\Vol(B_g(p, R))\geqslant \kappa R^{2n}$ for all $R>0$. 
 		\item Quadratic curvature decay:  there exists a $C>0$ such that $|Rm_g(q)|\leqslant Cd_g(p, q)^{-2}$ for all $q\neq p$. 
 	\end{itemize}
 We also denote by  $(\mathcal C, O, g_\mathcal C, J_\mathcal C, \omega_\mathcal C, \Omega_{\mathcal C})$ a Calabi-Yau cone. This means that  $O$ is a  
  distinguished point called the \emph{vertex}, $\mathcal C\setminus \{O\}$ is a smooth manifold diffeomorphic to $\R_{>0}\times L$ for some compact manifold $L$,   $(g_\mathcal C, J_{\mathcal C}, \omega_{\mathcal C}, \Omega_{\mathcal C})$ is a Calabi-Yau metric on $\mathcal C\setminus \{O\}$ and $g_\mathcal C=dr^2+r^2g_L$ is a Riemannian cone. A Calabi-Yau cone is naturally a normal affine algebraic cone.
 
It is known by \cite{DS2, Liu} that any $X$ as above is  asymptotic to a unique Calabi-Yau cone $\mathcal C$ at infinity. Moreover, $X$ is naturally a quasi-projective variety  and  there is an algebro-geometric 2-step degeneration from $X$ to $\mathcal C$, via a possible intermediate K-semistable cone $W$.  The main result of this paper is

 \begin{theorem}[No semistability at infinity]\label{t:main0}
 $W$ is  isomorphic to $\mathcal C$ as normal affine algebraic cones. 	
 \end{theorem}
 This comes out of our expectation, since the 2-step degeneration theory of \cite{DS2} works symmetrically with respect to asymptotic cones at infinity and local tangent cones at a singularity and it is known that $W$ \emph{can not} always be eliminated in the local setting. Theorem \ref{t:main0} reveals a sharp contrast between the two cases. Our proof makes a novel use of the Tian-Yau construction of complete Calabi-Yau metrics and the Bishop-Gromov volume monotonicity in Riemannian geometry. 
As a byproduct, we also show that 	$X$ is polynomially asymptotic to $\mathcal C$ at infinity:

\begin{theorem}[Polynomial rate to the asymptotic cone]
	\label{t:main2}
 There exists a diffeomorphism $\Psi$ from the open subset $\{r>1\}\subset \mathcal C$ to the complement of a compact set in $X$ such that for some $\delta>0$ and for all $k\geqslant0$,
 	\begin{equation}\label{e:polynomial rate}
 		\left|\nabla_{g_\mathcal C}^k(\Psi^*g-g_{\mathcal C})\right|_{g_{\mathcal C}}+\left|\nabla_{g_\mathcal C}^k(\Psi^*J-J_{\mathcal C})\right|_{g_{\mathcal C}}
 		+\left|\nabla_{g_\mathcal C}^k(\Psi^*\Omega-\Omega_{\mathcal C})\right|_{g_{\mathcal C}}=O(r^{-\delta-k}), \ \ \ r\rightarrow\infty.
 	\end{equation}
\end{theorem}

Theorem \ref{t:main2} bridges the gap between two previous results. On the one hand,  on the level of Riemannian metrics the  result of Colding-Minicozzi \cite{CM}, which uses the \L{}ojasiewicz-Simon theory, implies that $X$ is asymptotic to a unique cone at a \emph{logarithmic} rate.  On the other hand, Conlon-Hein \cite{CH3} classified complete Calabi-Yau manifolds asymptotic to a Calabi-Yau cone, however under the stronger extra assumption of  \emph{polynomial} rate (as  given exactly by \eqref{e:polynomial rate}).

\subsection*{Acknowledgements} The first author would like to thank Xuemiao Chen for many discussions on related questions in the context of Hermitian-Yang-Mills connections. The second author thanks Ronan Conlon for his lectures at UC Berkeley. We thank Lorenzo Foscolo for helpful discussions, and  Ronan Conlon, Yuji Odaka for comments. We also thank the anonymous referee for useful suggestions which improved the presentation. Both authors are partially supported by NSF Grant DMS-2004261. The first author is partially supported by the Simons Collaboration Grant in special holonomy.
\section{Outline of the proof}\label{two step degeneration}
\subsection{2-step degeneration theory} 
Fix $\lambda\equiv1/\sqrt{2}$. Denote by $(X_j, J_j, g_j, \Omega_j, p_j)$ the rescaling of $X$ by a factor $\lambda^{j}$. This means that we take $J_j=J, p_j=p, g_j=\lambda^{2j}g, \Omega_j=\lambda^{nj}\Omega$. It follows from the Bishop-Gromov volume monotonicity that after passing to a subsequence $(X_j, p_j)$ converges in the pointed Gromov-Hausdorff sense to a Calabi-Yau cone. We summarize below the theory developed in \cite{DS2} and \cite{Liu}. In  \cite{DS2} a more general singular setting is considered but under a technical assumption that $X$ is  a rescaled Gromov-Hausdorff limit of \emph{polarized} K\"ahler-Einstein manifolds. This assumption was later removed in \cite{Liu} for complete Calabi-Yau manifolds with Euclidean volume growth and quadratic curvature decay, so the results of \cite{DS2} apply in our setting. 

One conclusion of \cite{DS2, Liu} is that  the asymptotic cone is unique as a Calabi-Yau cone, i.e., it does not depend on the choice of subsequences involved in the definition. This may also follow from the more general result of Colding-Minicozzi \cite{CM} (which does not concern the complex structure), but the complex geometric proof in \cite{DS2, Liu} more importantly gives an algebro-geometric description of the asymptotic cone.  This is crucial for us so below we review some relevant statements.

Let $(\mathcal C, g_{\mathcal C}, J_{\mathcal C}, \omega_{\mathcal C}, \Omega_{\mathcal C}, O)$ be the unique asymptotic Calabi-Yau cone of $X$.  Since  $\mathcal C$ is a K\"ahler cone,  there is  a  Reeb vector field $\xi=J_{\mathcal C}(r\p_r)$, which is holomorphic and Killing. It generates a holomorphic isometric action of a   compact torus $\mathbb T$ on $\mathcal C$. 
    The K\"ahler form has a simple expression given by $\omega_{\mathcal C}=\frac{\ii}{2}\p_{J_\mathcal C}\bp_{J_\mathcal C} r^2$. Furthermore, $\mathcal C$ is naturally a normal affine variety.  The coordinate ring  $R(\mathcal C)$ can be intrinsically characterized as  the space of holomorphic functions on $\mathcal C\setminus \{O\}$ that have polynomial growth at infinity. More explicitly, 
there is a holomorphic embedding $\Phi_\infty: (\mathcal C, O)\rightarrow(\C^N, 0)$ such that the Reeb vector field $\xi$ extends to a linear vector field on $\C^N$ of the form $Re(\sqrt{-1}\sum_j d_j z_j\p_{z_j})$, where $d_j\in \mathbb R_{>0}$ for all $j$. 
	Let $\Lambda: \C^N\rightarrow\C^N$ be the diagonal linear transformation of $\C^N$ given by 
$$\Lambda(z_1,\cdots, z_N)=(\lambda^{d_1}z_1, \cdots, \lambda^{d_N}z_N), $$
and let $T_\xi$ be the one parameter group of linear transformations generated by $\xi$.
Denote by $G_\xi$ the subgroup of  $GL(N;\C)$ consisting of elements that commute with $T_\xi$ and denote $K_\xi=G_\xi\cap U(N)$. Notice $T_\xi$ naturally acts on $\mathcal C$.

Denote by $B_j$ the unit ball around $p_j$ in $X_j$, which we can identify with the ball of radius $\lambda^{-j}$ around $p$ in $X$. Denote by $B_\infty$ the unit ball around the vertex $O$ in $\mathcal C$. Then we know that $B_\infty$ is the Gromov-Hausdorff limit of $B_j$ as $j\rightarrow\infty$.  Fix a distance function on $B_j\sqcup B_\infty$ that realizes the Gromov-Hausdorff convergence. We have the following

\begin{enumerate}[(a).]
	\item  The ring $R(X)$ of holomorphic functions on $X$ with polynomial growth is finitely generated and $\text{Spec}(R(X))$ defines a normal affine algebraic variety $X'$ with isolated singularities. There is a natural map $\pi:(X,p)\rightarrow (X',p')$ which is a crepant resolution of singularities.
	\item For all $j\geqslant 1$, there are holomorphic embeddings $\Phi_j: (X',p')\rightarrow (\C^N,0) $ and $F_j\in G_\xi$ satisfies $F_j=\Lambda+\tau_j$ for linear maps $\tau_j\rightarrow 0$, such that 

\begin{enumerate}[(1).]
\item $\Phi_{j+1}=F_{j+1}\circ\Phi_{j}$.
	\item For any subsequence of $\{j\}$ tends to infinity, passing to a further subsequence $\Phi_j(\pi(B_j))$ converges to $h.\Phi_\infty(B_\infty)$ in the Hausdorff sense in $\C^N$ for some $h\in K_\xi$. This convergence is compatible with the Gromov-Hausdorff convergence in the sense that given any sequence $x_j\in B_j$ converging to $x_\infty\in B_\infty$, we have $\Phi_j(\pi(x_j))$ converges to $h.\Phi_\infty(x_\infty)$.
	\item  We will always identify $(X', p')$ with $(\Phi_1(X'),0)$. Denote $Y_j\equiv\Lambda^{j-1}.X'$. Then the weighted asymptotic cone $W=\lim_{j\rightarrow\infty}Y_j$ is a normal affine algebraic variety in $\C^N$. In particular, $W$ is also invariant under the $T_\xi$ action. This limit is to be understood in a strong sense: there exists a generating set $\{P_1, \cdots, P_l\}$ of the ideal $\mathcal I_{X'}$ of polynomials on $\C^N$ vanishing on $X'$, such that the leading terms $\{\tilde P_1, \cdots, \tilde P_l\}$ generate the ideal $\mathcal I_{W}$. Here the leading term of a polynomial is defined via the weighted degree associated to $\Lambda$.
	\item The coordinate ring $R(W)$ can be intrinsically described as the graded ring associated to a filtration on $R(X)$ determined by a \emph{degree} function associated to the Calabi-Yau metric $g$.
	\item 
	$\mathcal C$ is in the closure of the $G_\xi$ orbit of $W$ in a suitable multi-graded Hilbert scheme. In particular, $W$ has only one isolated singularity.
\end{enumerate}  
\end{enumerate}

The motivation in \cite{DS2} was to relate the metric scaling on $X$ to the algebro-geometric scaling $\Lambda$. The main discovery there is that they are \emph{almost but not exactly} the same, and one
 may realize the asymptotic cone $\mathcal C$ as a 2-step degeneration from $X'$ through the intermediate cone $W$. From the technical aspect the central issue is that as $j\rightarrow\infty$, even though $F_j$ converges to $\Lambda$ the accumulated error in the composition $F_j\circ F_{j-1}\circ \cdots \circ F_2$ may still diverge away from $\Lambda^{j-1}$.   Notice $W$ does not inherit any canonical metric from $X$ or $\mathcal C$. In 
\cite{DS2}  a parallel result is also proved for tangent cones at local singularities of K\"ahler-Einstein metrics. However there are important distinctions between the two settings.

One such distinction was already pointed out in \cite{DS2}, which shows that the local setting is more \emph{rigid} and the asymptotic setting is more \emph{flexible}. In the case of local tangent cones, it is  conjectured in \cite{DS2} that there is a local notion of (K-)stability which characterizes both $(W, \xi)$ and $\mathcal C$ as \emph{invariants} of the underlying algebraic singularity (by analogy with the Harder-Narasimhan-Seshadri filtration for sheaves, one should think of $W$ as a \emph{semistable} object and $\mathcal C$ as a \emph{polystable} object). Later Li \cite{ChiLi} related this picture to a generalized volume minimization in Sasaki geometry and reformulated this conjecture using more algebro-geometric terminologies. 
There has been active research  on this topic. The uniqueness part of the conjecture in \cite{DS2} was confirmed by Li-Wang-Xu \cite{LWX}. By way of contrast as suggested at the end of \cite{DS2}, for asymptotic cones of complete Calabi-Yau metrics one should not expect either $W$ or $\mathcal C$ is a canonical object associated to the underlying algebraic variety $X$. Explicit examples have been constructed on $\C^n$ if we allow $W$ and $\mathcal C$ to have singular cross sections (see \cite{HN, YLi, Gabor1, CR}).

The main result of this paper shows that the asymptotic setting is however more \emph{rigid} from a different perspective. 
In the case of tangent cones at singularities of K\"ahler-Einstein metrics, it is known that both of the 2 steps are necessary in order to  degenerate the singularity to the metric tangent cone $\mathcal C$.   In the case of asymptotic cones at infinity we will  show that $W$ is always isomorphic to $\mathcal C$, in other words, one can degenerate the variety to $\mathcal C$ in one step.  Theorem \ref{t:main0} can be reformulated as

\

 \noindent {\bf Theorem \ref{t:main0}} (No semistability at infinity).
\emph{  $W$ is isomorphic to $\mathcal C$ as affine varieties with $T_\xi$ action.   }

\subsection{Outline of the proof}
Now we explain the key ideas in the proof of Theorem \ref{t:main0}. Theorem \ref{t:main2} will  be a consequence.
Our strategy is to construct a Calabi-Yau metric ${\omega_W}$ on $W\setminus\{0\}$ so that 
\begin{enumerate}
	\item the asymptotic cone of ${\omega_W}$ at infinity is given by $\mathcal C$;
  	\item the metric completion of ${\omega_W}$ is homeomorphic to $W$ and is a pointed Gromov-Hausdorff limit of complete Calabi-Yau metrics. Furthermore  the argument in \cite{DS2} can be extended to study the singular point $0\in W$.
\end{enumerate}
Now the Bishop-Gromov inequality appling on $W$ yields that the volume ratio $$\nu(R)=\frac{\text{Vol}(B(0, R))}{R^{2n}}$$ is a non-increasing function in $R$. On a Calabi-Yau cone $\mathcal C$ we denote by $\kappa(\mathcal C)=\text{Vol}(B(O, 1))$ the volume density of $\mathcal C$.   Now as $R\rightarrow\infty$,  $\nu(R)$ converges to $\kappa(\mathcal C)$, while as $R\rightarrow0$ it converges to the volume density of the metric tangent cone at $0$. By \cite{DS2} and the generalized volume minimization in \cite{LX} we know the latter is not bigger than $\kappa(\mathcal C)$.   So  $\nu(R)$ is constant in $R$, which implies that $(W, {\omega_W})$ must be a metric cone hence is isometric to $\mathcal C$. With  more work one can show that $W$ and $\mathcal C$ are $T_\xi$-equivariantly isomorphic.

 The existence of ${\omega_W}$ is therefore the crux of the matter. To achieve this we will make use of the Tian-Yau construction \cite{TY2}  which produces complete Calabi-Yau metrics out of lower dimensional compact K\"ahler-Einstein metrics. In \cite{TY2} one starts with a Fano manifold $X$ and a smooth divisor $D$ which is $\mathbb Q$-linearly equivalent to $K_X^{-1}$, then assuming that $D$ admits a K\"ahler-Einstein metric with positive Ricci curvature, one can find a complete Calabi-Yau metric on $X\setminus D$. The rough idea is to first use the Calabi ansatz to write down a background complete K\"ahler metric  which is approximately Ricci-flat and then solve the complex Monge-Amp\`ere equation on the noncompact manifold $X\setminus D$. The work of \cite{TY2} has been extended widely to produce more examples of Calabi-Yau metrics. 

 In the proof of Theorem \ref{t:main0} we will apply the Tian-Yau construction in a non-traditional manner. That is, we will use the unknown Calabi-Yau metric $\omega$ on $X$ as a model at infinity to construct a Calabi-Yau metric ${\omega_W}$ on $W$. This is achieved by constructing a diffeomorphism between the ends of $X$ and  $W$,  then grafting $\omega$  to a K\"ahler form on $W$ that is slowly asymptotic to the cone $\mathcal C$, but satisfies the complex Monge-Amp\`ere equation with polynomially decaying error. Most of the work in this paper is devoted to making this strategy rigorous. There is an extra  difficulty caused by the fact that $W$ has an isolated singularity. To get around the issue, we notice that $W$ is the limit of $Y_j$ in $\C^N$, and $Y_j$ admits a crepant resolution $\hat Y_j$ (which is isomorphic to $X$). We will   first construct a family of Calabi-Yau metrics $\hat\omega_{j, \epsilon}$ on $\hat Y_j$ with uniform estimates, where $\epsilon$ roughly measures the size of the exceptional set,  and then take a double limit by first letting $\epsilon\rightarrow 0$ and then letting $j\rightarrow\infty$. The first limiting step follows similarly the line of argument in \cite{CGT} which constructs certain asymptotically conical (in the  sense of \cite{CH1}) Calabi-Yau metrics with singularities, under more restrictive assumptions.  In the course of the proof we will bring in more robust techniques, such as the H\"ormander $L^2$ estimate, which is of independent interest.

In the rest of this paper we will present a detailed proof of the main results. 
In Section \ref{S3} we construct a family of background K\"ahler metrics on $\hat Y_j$ which are asymptotic to the cone $\mathcal C$ and are approximate solutions to the complex Monge-Amp\`ere equation at infinity. In Section \ref{s5} we solve the  complex Monge-Amp\`ere equation  with respect to these background metrics and derive uniform estimates. In Section \ref{s6} we complete the proof of Theorem \ref{t:main0} and \ref{t:main2}.

We make a convention on the notations throughout this paper: $K$ will denote a compact set in some space; $C$ will denote a  positive constant; $C_{\sharp}$ will denote a  constant that depends only on $\sharp$; in particular when $j$ does not appear in the subscript it means that the constant is uniform in $j$. The precise meaning of these objects may vary between lines.

\section{Construction of background metrics} \label{S3}

\subsection{Weak conical approximation at infinity} We denote by $r=d(O, \cdot)$ the radial function on the asymptotic Calabi-Yau cone $\mathcal C$, by $\mathcal Q_R$ the subset $\{r\geqslant R\}$ and by $A_{a, b}$ the annulus $\{a\leqslant r\leqslant b\}$. The following lemma is standard so we only give a sketch of proof. 
\begin{lemma}\label{l:asymptotic decay}
 There exist an $R_0>0$, a compact set $K\subset X$ and a diffeomorphism $\Phi: \mathcal Q_{R_0}\rightarrow X\setminus K$ such that for all $l\geqslant 0$, $\lim_{s\rightarrow\infty}\sup_{\p \mathcal Q_s}s^l\left (\left|\nabla^l_{g_\mathcal C}(\Phi^*g-g_{\mathcal C})\right|_{g_\mathcal C}+\left |\nabla^l_{g_\mathcal C}(\Phi^*J-J_{\mathcal C})\right |_{g_\mathcal C}\right)=0$. 
\end{lemma}
\begin{proof}
Since $\mathcal C$ is the unique asymptotic cone of $X$, for $i$ large we may find a diffeomorphism $\Phi_i$ from $A_{4^{i}, 8\cdot 4^{i}}$ onto $B(p, 8\cdot 4^i(1+\epsilon_i))\setminus B(p, 4^i(1+\epsilon_i))$ such that  $$4^{il}\sup_{A_{4^i, 8\cdot 4^i}}\left(\left|\nabla^l_{g_\mathcal C}(\Phi_i^*g-g_{\mathcal C})\right|_{g_\mathcal C}+\left|\nabla^l_{g_\mathcal C}(\Phi_i^*J-J_{\mathcal C})\right|_{g_\mathcal C}\right)\leqslant\epsilon_{i,l},$$ where $\epsilon_i\rightarrow0$ and for each $l$ fixed, $\epsilon_{i, l}\rightarrow 0$ as $i\rightarrow\infty$. Moreover, the rescaled maps $\Phi_i\circ 4^i$ realize the Gromov-Hausdorff convergence to the cone $\mathcal C$.  It suffices to glue together $\Phi_i$'s to get the desired $\Phi$. This can done in a  straightforward way. By a contradiction argument one can prove that for $i\geq i_0$, in the intersection  $\text {Im}(\Phi_i)\cap\text{Im}	(\Phi_{i+1})$, the composition $4^{-i-1}\circ \Phi_{i+1}^{-1}\circ \Phi_i\circ 4^{i+1}$ is $\epsilon_i$ (for $\epsilon_i\rightarrow 0$) close to a holomorphic isometric embedding $\Psi_i$ of the annulus $A_{1,2}$ into itself with $\Psi_i^*r=r$. It follows that $\Psi_i$ preserves the Reeb vector field on $\mathcal C$, hence it extends to a global holomorphic isometry of $\mathcal C$ which fixes the vertex. Then we replace $\Phi_{i+1}$ by $\Phi_{i+1}\circ \Psi_i\circ\cdots\circ \Psi_{i_0}$ and perform an obvious interpolation between $\Phi_i$ and $\Phi_{i+1}$.
\end{proof}

Notice that at this point  we know the metric $\Phi^*g$ is asymptotic to $g_{\mathcal C}$ at infinity, but without any quantitative \emph{rate} of decay. Nevertheless, the difference of the Levi-Civita connections does decay at the order $o(r^{-1})$. The general result of Colding-Minicozzi \cite{CM} yields a logarithmic decay  rate of the Riemannian metric but we will not need it in this paper.  Using the diffeomorphism constructed above, we may naturally view $r$ as a function on the end of $X$ which satisfies
$$C^{-1} d_g(p, \cdot)\leqslant r(\cdot)\leqslant C d_g(p, \cdot).$$
Given a constant $\widetilde R> R_0+1$, we can further replace $r$ by the regularized maximum $\widetilde{\max}\{r,\widetilde R\}$ (see \cite[Chapter 1,Lemma 5.18]{Demailly}), 
so it can be viewed as a smooth function on $X$. Moreover, by Lemma \ref{l:asymptotic decay}, we can fix an  $\widetilde R$ large so that $r^2$ is a smooth plurisubharmonic exhaustion function on $X$ which is strictly plurisubharmonic outside a compact set. Then we know that for any sufficiently large number $R$, the sublevel set 
\begin{equation}\label{sublevel set}
    X_R :=\{r < R\}
\end{equation}
is 1-convex. By the Grauert-Riemenschneider vanishing theorem \cite{GR} and triviality of the canonical bundle $K_X$, we get
\begin{equation}\label{vanishing of H^1}
    H^1(X_R,\mathcal O)=H^1(X_R,K_{X_R})=0.
\end{equation}
This will be used in the proof of Theorem \ref{decomposition of kahler forms}.
\subsection{Weighted elliptic estimates}
Here we collect some facts on weighted elliptic theory for  ends of Riemannian manifolds asymptotic to cones. These are not sharp results but suffice for our applications in this paper. 

 For any $k\in \mathbb Z_{\geqslant 0}$, $\alpha\in (0, 1)$, $\delta\in \mathbb R$, $R>0$, we define  the Banach space $C^{k, \alpha}_{\delta}(\mathcal Q_R)$ using the following weighted norm on a function $f$ defined on $\mathcal Q_R\subset\mathcal C$:
\begin{align*}
\|f\|_{C_{\delta}^{k, \alpha}(\mathcal \mathcal Q_R)}& \equiv \sup_{s\geqslant R}\Big\{ \sum\limits_{m=0}^ks^{-\delta+m} \|\nabla_{g_{\mathcal C}}^m f\|_{C^0_{g_\mathcal C}(A_{s,2s})} +  s^{-\delta+k+\alpha}[f]_{C^{k,\alpha}_{g_{\mathcal C}}(A_{s,2s})}\Big\},\end{align*}
where 
\begin{align*}
[f]_{C^{k,\alpha}(A_{s,2s})} \equiv \sup\limits_{y_1, y_2\in A_{s,2s}, d_{g_\mathcal C}(y_1,y_2)<\text{inj}_{g_\mathcal C}(y_1)}\bigg\{\frac{|\nabla^k_{g_\mathcal C}f(y_1)-\nabla^k_{g_\mathcal C}f(y_2)|}{d_{g_\mathcal C}(y_1,y_2)^\alpha} \bigg\}.
\end{align*}
Fix $k\geqslant 2n+1$ and $\alpha\in (0, 1)$.   
 The following can be proved in exactly the same way as \cite[Proposition 6.7]{SZ}, so we omit the proof here. 
\begin{lemma}\label{p:weighted-esimate-quotient-space}
 There exists a finite set $\Gamma\subset(0, 1)$ such that for all $\delta\in (0, 1)\setminus \Gamma$ and for all $R\geqslant 1$, one can find a bounded linear map $\mathcal S_R:  C^{k, \alpha}_{-\delta}(\mathcal Q_R)\to C^{k+2, \alpha}_{-\delta+2}(\mathcal Q_R)$ with  $\Delta_{g_\mathcal C}\circ \mathcal S_R=\Id$ and with the operator norm $\|\mathcal S_R\|\leqslant C$ for some $C$ depending only on $\mathcal C$, $\delta$, $k$, $\alpha$ (but not on $R$). 
\end{lemma}

Given any $R_1>1$ and suppose  there is another Riemannian metric $g$ on $\mathcal Q_{R_1}$ such that for all $l\geqslant 0$ and $s\geqslant R_1$,
\begin{equation}\label{new riemannian metric}
    s^l\sup_{\p \mathcal Q_s}|\nabla^l_{g_{\mathcal C}} (g-g_{\mathcal C})|\leqslant e_l(s)
\end{equation}
 for some function $e_l$ defined on $[R_1, \infty)$  with $\lim_{s\rightarrow \infty}e_l(s)=0$.

\begin{proposition} \label{p:Poisson equation general}
Given $\delta\in (0,1)\setminus \Gamma$, there exist $R_2\geqslant R_1$ and $C>0$ depending only on $\mathcal C$, $R_1$, $k, \delta, \alpha$ and the functions $e_1,\cdots, e_{k+2}$, and a bounded linear map $\mathcal T: C^{k,\alpha}_{-\delta}(\mathcal Q_{R_2})\to C^{k+2, \alpha}_{-\delta+2}(\mathcal Q_{R_2})$ such that $\Delta_{g}\circ \mathcal T=\Id$ and $\|\mathcal T\|\leqslant C$. \end{proposition}
\begin{proof}
One can write $\Delta_{g}=\Delta_{g_\mathcal C}+(g-g_{\mathcal C})*\nabla^2_{g_\mathcal C}+\nabla_{g_{\mathcal C}}g*\nabla_{g_\mathcal C}$. Then on $\mathcal{Q}_R$, $R\geqslant R_1$, we have
$$\|\Delta_g-\Delta_{g_\mathcal C}\|\leqslant C  \sup_{s\geqslant R, \ 0\leqslant l \leqslant k+2}e_l(s).$$
The right hand side is small when $R\gg1$. Then the conclusion follows from standard functional analysis. \end{proof}

The following is an application of the standard elliptic regularity by a straightforward rescaling argument. We omit the proof. 

\begin{proposition}\label{c:elliptic regularity} Let $g$ be a Riemannian metric on $\mathcal{Q}_{R_1}$ satisfying \eqref{new riemannian metric}.
	Suppose $u$ is a function defined on $\mathcal Q_{R_1}$ such that  for some $s, s'\in \mathbb R$ and for all $l\geqslant0$, $\left|\nabla^l_g(\Delta_gu)\right|_g\leqslant C_l r^{-l-2+s}$, and $\int_{\mathcal{Q}_{R_1}}u^2r^{-2s'-2n}d\operatorname{Vol}_g\leqslant \hat C$. Then for any $l>0$ and $\gamma>\max(s, s')$,  there exist constants $A_{\gamma l}>0$ depending only on $\gamma,\mathcal C$, $R_1$, $C_0, C_1, \cdots, C_l, \hat C$ and the functions $e_1,\cdots, e_{l}$, such that
	\begin{equation*}
	    	\left|\nabla_g^{l}u\right|_g\leqslant A_{\gamma l} r^{\gamma-l} \text{ on $\mathcal{Q}_{R_1}$}.
	\end{equation*}
\end{proposition}

\subsection{Rough algebraic approximation of the Calabi-Yau metric}
  We continue to use the notations in Section \ref{two step degeneration}. Recall that $X'$ is the affine variety defined by the ring $R(X)$ of holomorphic functions on X with polynomial growth and we identify it with its image in $\mathbb C^N$ under the holomorphic embedding $\Phi_1$. We also have the crepant resolution map $\pi: X\rightarrow X'\subset \mathbb C^N$. In the following outside a compact $K$ of $X$ containing the exceptional set of $\pi$, we also identify $X$ with $X'$. In particular, $X\setminus K$ is holomoprhically embedded in $\mathbb C^N$.
  
  Since the coordinate functions on $\C^N$ have positive weight with respect to $\xi$, we can find a K\"ahler cone metric $\omega_\xi$ on $\C^N$ with Reeb vector field $\xi$. Fix a choice of such  $\omega_\xi$. Then we have $\Lambda^*\omega_\xi=\lambda^{2}\omega_\xi$.  Denote by $r_\xi=d_{\omega_\xi}(0, \cdot)$ the radial function on $\C^N$ defined by $\omega_\xi$. Notice we are not assuming $\omega_{\xi}$ is Calabi-Yau. The following result shows that we may view the restriction of $\omega_\xi$ to $X$ as a rough approximation of the unknown Calabi-Yau metric $\omega$.
 \begin{proposition}\label{rough estimate}
There exists a compact set $K\subset X$ such that for all $\delta>0$ and $k\geqslant 1$, on $X\setminus K$ we have	\begin{equation}\label{e:rough comparison} C_\delta^{-1} r^{-\delta}\omega\leqslant \omega_\xi\leqslant C_\delta r^{\delta}\omega, 
 \end{equation}
 \begin{equation}\label{eqn4.3}\left|\nabla^k_{\omega}\omega_{\xi}\right|_{\omega}\leqslant C_{\delta,k} r^{\delta-k}, 
 \end{equation}
 \begin{equation}\label{e:rough comparison of distance}
 C_\delta^{-1} r^{1-\delta}\leqslant r_\xi\leqslant C_\delta r^{1+\delta}.	
\end{equation}
 \end{proposition}

\begin{proof}Since we identify $X'$ with $\Phi_1(X')$ and outside a compact set identify $X$ with $X'$  using $\pi$, we have $\Phi_j=F_j\circ \cdots F_2:X\rightarrow \mathbb C^N$. By Item (b) in Section \ref{two step degeneration} we know that  for all $j\geqslant 1$ and $k\geqslant 0$, the following holds on $B_{j+1}\setminus B_{j}$:
\begin{equation}\label{e:compare omega and omegaxi}
C^{-1}\lambda^{2j}r\leqslant \Phi_j^*r_\xi\leqslant C \lambda^{2j}r,  	\ \ \ \ 
C^{-1}\lambda^{2j}\omega\leqslant \Phi_j^*\omega_\xi\leqslant C \lambda^{2j}\omega,  	\ \ \ \ |\nabla^k_{\omega}(\lambda^{-2j}\Phi_j^*\omega_{\xi})|_{\omega}\leqslant C_{k} \lambda^{kj}.	
\end{equation}
By construction we have $F_j^{-1}=\Lambda^{-1}\circ \sigma_j$ for a linear map $\sigma_j\in G_\xi$ with $\|\sigma_j-\Id\|\rightarrow0$, where $\|\cdot\|$ is a fixed norm on the space of $N\times N$ matrices. Since $\sigma_j$ commutes with $\Lambda$ we may write 
$$F_2^{-1}\circ \cdots\circ F_j^{-1}=\Lambda^{1-j}\sigma_2\circ \cdots\circ \sigma_j. $$
For any $\sigma\in G_\xi$  we have $\mathcal L_{r_\xi\p_{r_\xi}}\sigma^*\omega_\xi=2\sigma^*\omega_\xi$. It follows that $\sigma^*\omega_\xi$ is homogeneous of degree 2 on $\C^N$. So one can find a small neighborhood $U$ of $\text{Id}$ such that for all $\sigma\in U$ and all $k\geqslant 0$, the following holds on $\C^N$:
\begin{equation}\label{e:pull back metric estimate}|\nabla^k_{\omega_\xi}(\sigma^*\omega_\xi-\omega_\xi)|_{\omega_\xi} \leqslant C_k r_\xi^{-k}\|\sigma-\Id\|.\end{equation}
In particular we can also assume
\begin{equation}\label{e:connection difference estimate}|\nabla^k_{\omega_\xi}(\nabla_{\sigma^*\omega_\xi}-\nabla_{\omega_\xi})|_{\omega_\xi}\leqslant C_k r_\xi^{-k-1}\|\sigma-\Id\|, 
\end{equation}
where as usual we view the difference of the two connections as a tensor. 

Since $\|\sigma_j-\Id\|\rightarrow 0$, for $j\gg1$ we have $\sigma_j\in U$.  
We notice the elementary fact that for all $\delta>0$,  $$\prod_{l=2}^j\big(1+C\|\sigma_l-\Id\|\big)\leqslant C_\delta \lambda^{-j\delta}.$$ From this, iterating \eqref{e:pull back metric estimate} and using \eqref{e:compare omega and omegaxi} yield  \eqref{e:rough comparison}.  \eqref{e:rough comparison of distance} follows from a similar argument. 

We now prove the case $k=1$ in \eqref{eqn4.3}. The case $k>1$ can be proved similarly by induction. 
Notice for any $\sigma\in U$ and 2-form $\gamma$, we have 
\begin{eqnarray*}
\nabla_{\omega_\xi}\sigma^*\gamma&=&\nabla_{\sigma^*\omega_\xi}\sigma^*\gamma+(\nabla_{\omega_\xi}-\nabla_{\sigma^*\omega_\xi})\sigma^*\gamma\\&=&\sigma^*(\nabla_{\omega_\xi}\gamma)+(\nabla_{\omega_\xi}-\nabla_{\sigma^*\omega_\xi})\sigma^*\gamma.
\end{eqnarray*}
Therefore,
$$|\nabla_{\omega_\xi}(\sigma_j^*\cdots \sigma_2^*\omega_\xi)|_{\omega_\xi}=|\sum_{l=2}^{j} \sigma_j^*\cdots \sigma_{l+1}^*(\nabla_{\omega_\xi}-\nabla_{\sigma_l^*\omega_\xi})\sigma_{l-1}^*\cdots\sigma_2^*\omega_\xi|_{\omega_\xi}.$$
Using \eqref{e:pull back metric estimate} and \eqref{e:connection difference estimate} we obtain that on the annulus $\{\lambda\leqslant r_\xi\leqslant \lambda^{-1}\}\subset\C^N$ for all $\delta>0$, 
\begin{equation}\label{eqn4.7}|\nabla_{\omega_\xi}(F_2^{-1}\circ\cdots\circ F_j^{-1})^*\omega_\xi|_{\omega_\xi}\leqslant C \lambda^{-2j}\prod_{l=2}^j (1+C\|\sigma_l-\Id\|)\leqslant C_\delta \lambda^{-(2+\delta)j}. 
\end{equation}
By definition $\omega_\xi=\Phi_j^*\left((F_2^{-1}\circ\cdots\circ F_j^{-1})^*\omega_\xi\right)$, so  we have 
$$|\nabla_{\omega}\omega_\xi|_\omega=|\Phi_j^*(\nabla_{\omega_\xi}(F_2^{-1}\circ\cdots\circ F_j^{-1})^*\omega_\xi)|_\omega+|(\nabla_{\omega}-\nabla_{\Phi_j^*\omega_\xi})\Phi_j^*(F_2^{-1}\circ\cdots\circ F_j^{-1})^*\omega_\xi|_\omega.$$
The first term of the right hand side can be handled using \eqref{eqn4.7} and \eqref{e:compare omega and omegaxi}, and the second term  can be dealt with using \eqref{e:compare omega and omegaxi} and \eqref{e:rough comparison}. Together we get \eqref{eqn4.3} when $k=1$.
\end{proof}

      Although not needed in this paper, we remark that the above $\omega_\xi$ is not unique and we can find one with explicit algebraic formula.  We emphasize that from the proof we can not get rid of the $\delta$ in the estimates. In particular, we can \emph{not} claim the uniform equivalence between $\omega$ and $\omega_\xi$. Nonetheless we do know the error is smaller than any polynomial order which suffices for our applications. Moreover, the differences between the covariant derivatives do decay at infinity. Namely,

\begin{corollary}\label{estimates for different covariant derivatives}
 There exists a compact set $K\subset X$ such that on $X\setminus K$, \text{for all $\delta>0$ and $k\geqslant 0$} we have
       \begin{equation}\label{estimate for covariant derivatives}
       |\nabla^k_{\omega}(\nabla_\omega-\nabla_{\omega_{\xi}})|_{\omega}\leqslant C_{\delta,k} r^{\delta-k-1}.
 \end{equation}
\end{corollary}

\begin{proof}
Two Riemannian metrics $g_1$ and $g_2$ on a manifold determine a section $T$ of $\operatorname{End}(TM)$ via 
\begin{equation*}
    g_2(\cdot,\cdot)=g_1(T\cdot,\cdot), 
\end{equation*}
which is positive and self-adjoint with respect to both $g_1$ and $g_2$. Using Koszul formula, one can check directly that
\begin{equation*}
    \nabla_{g_2}-\nabla_{g_1}=T^{-1}*\nabla_{g_1}T
\end{equation*} where we view both sides as (1,2)-tensors and $*$ denotes some algebraic operations. Therefore  to estimate $|\nabla_{g_1}^k(\nabla_{g_2}-\nabla_{g_1})|_{g_1}$ it suffices to estimate $T^{-1}$ and $\nabla^k_{g_1}T$. Then \eqref{estimate for covariant derivatives} follows from Proposition \ref{rough estimate}.
\end{proof}

\subsection{Plurisubharmonic weight functions}
We construct here some plurisubharmonic functions that will be used later. As mentioned in Section \ref{two step degeneration} we always identify $(X',p')$ with $(\Phi_1(X'),0)$ in $\mathbb C^N$. Then we can naturally restrict functions and forms defined on $\C^N$ to $X'$ and then pull-back to $X$ via the crepant resolution map $\pi$. When there is no confusion we will omit the various pull-back notations. 
 \begin{lemma}\label{psh function on a kahler cone}
     (1). Let $(\mathcal K, \omega_{\mathcal K})$ be a K\"ahler cone and let $\rho$ denote the distance function from the vertex of the cone. Then the function $\varphi  =\log(\rho^2+1)-\frac{1}{100}(\log(\rho^2+3))^{\frac{1}{2}}$ is a strictly plurisubharmonic function and there exists a compact set $K\subset \mathcal K$ such that on  $\mathcal K\setminus K$,\begin{equation*}
	\ii\partial\pp \varphi\geqslant C \rho^{-2}(\log \rho)^{-\frac{3}{2}}\omega_{\mathcal K} .
\end{equation*}

(2). The function on $X$  defined by $\varphi=\log(r_{\xi}^2+1)-\frac{1}{100}(\log(r_{\xi}^2+3))^{\frac{1}{2}}$ satisfies $\sqrt{-1}\p\bp\varphi>0$ away from the exceptional set of $\pi$. Moreover outside a compact set $K$ it satisfies that for all $\epsilon>0$, $$(1-\epsilon)\log (r^2+1)-C_{\epsilon}\leqslant \varphi\leqslant (1+\epsilon)\log (r^2+1)+C_{\epsilon},$$ $$\ii\partial\pp \varphi\geqslant {C_{\epsilon}^{-1}}{r^{-2-\epsilon}}\omega,$$ 
		$$|\nabla_\omega\varphi|_{\omega}\leqslant C_\epsilon r^{-1+\epsilon}. $$
 \end{lemma}
 \begin{proof}
   (1) follows from a straightforward computation: the first term is strictly  plurisubharmonic but is decaying at the rate $\rho^{-4}$ along the radial direction; the second term is used to strengthen the positivity along the radial direction. Applying (1) to the K\"ahler cone $(\C^N, \omega_\xi)$ and then using Proposition \ref{rough estimate}, we obtain (2).
 \end{proof}

\subsection{Decomposition of the  K\"ahler form}
We will use the following version of H\"ormander's $L^2$ estimate, which follows from for example \cite[Chapter VIII, Theorem 6.1]{Demailly}.

\begin{theorem}\label{Hormander}
    Let $M$ be an $n$-dimensional complex manifold admitting a complete K\"ahler metric and $\omega_M$ be a K\"ahler metric on $M$ which is not necessarily complete. Suppose $\varphi_M$ is a smooth function with $\sqrt{-1}\p\bp\varphi_M\geqslant \Upsilon\omega_M$ for a continuous non-negative function $\Upsilon$. Let $q$ be a positive integer. Then we have 
    \begin{itemize}
        \item [(1).] for any  $(n,q)$ form $\eta$ on $M$  with  $\bp \eta=0$ and $\int_{M}\Upsilon^{-1} |\eta|_{\omega_M}^2e^{-\varphi_M}\omega_M^n<\infty$, there exists an $(n, q-1)$ form  $\zeta$ satisfying $\bp\zeta=\eta$ and with estimate
    \begin{equation*}
        \int_{M} |\zeta|_{\omega_M}^2e^{-\varphi_M}\omega_M^n\leqslant \int_{M}q^{-1}\Upsilon^{-1}|\eta|_{\omega_M}^2e^{-\varphi_M}\omega_M^n;
    \end{equation*}
    \item[(2).] if the Ricci curvature of $\omega_M$ is non-negative, then the conclusion in (1) holds for $(0,q)$ forms.
    \end{itemize}

\end{theorem}

The main result of this subsection is 
\begin{theorem}\label{decomposition of kahler forms}
	 For every $\epsilon>0$, there exist a K\"ahler form $\beta$ and a real-valued smooth function $\psi$ on $X$ such that 
\begin{enumerate}[(1).]
		\item $\omega=\beta+\ii\partial\pp \psi$,
		\item $|\nabla^k_{\omega}\beta|_{\omega}=O(r^{-2+\epsilon-k})$ for all $k\geqslant 0$,
		\item $|\nabla_{\omega}^k\psi|_{\omega}=O(r^{2+\epsilon-k})$ for all $k\geqslant 0$.
\end{enumerate}
\end{theorem}
 
 \begin{proof}
 Let $\varphi$ be the plurisubharmonic function on $X$ obtained in Lemma \ref{psh function on a kahler cone}-(2). Fix $\epsilon>0$.
 
 \
 
 \textbf{Step 1.}   Using the diffeomorphism $\Phi$ constructed in Lemma \ref{l:asymptotic decay}, for  $R$ large we can identify $X\setminus X_R$ with $\mathcal{Q}_R$ and we may assume $X_R$ is 1-convex and contains the exceptional set of $\pi$. Then one can write $\omega=\tilde\beta+d\tilde\eta$ with $|\nabla^k_{\omega} \tilde\beta|_{\omega}=O(r^{-2-k})$ for all $k\geqslant 0$ and $\widetilde \eta=0$ on $X_R$.

 \
 
 \textbf{Step 2.}
 We decompose $\tilde\beta$ according to types:
 	$
 		\tilde\beta=\tilde \beta^{2,0}+\tilde \beta^{1,1}+\tilde \beta^{0,2}.
 $
 Since $d\tilde \beta=0$, we have $\bp\tilde\beta^{0, 2}=0$.  Notice by \textbf{Step 1}, we know that for all $\tau>0$, $$\int_{X} |\tilde\beta^{0,2}|_\omega^2 e^{-(n-2+\tau)\varphi}\omega^n<\infty.$$
Notice also that $\widetilde{\beta}^{0,2}=0$ on $X_R$, so we may apply Theorem  \ref{Hormander} (with $\varphi_M=(n-1+\tau)\varphi$) to obtain that for every $\tau>0$, there exists a (0,1)-form $\zeta$ on $X$ satisfying $\bp\zeta=\tilde\beta^{0,2}$ and $\int_{X}|\zeta|_{\omega}^2 e^{-(n-1+\tau)\varphi}\omega^n<\infty$.
It follows from  Proposition \ref{c:elliptic regularity} that $|\nabla_{\omega}^k\zeta|_{\omega}=O(r^{-1+2\tau-k})$ for all $k\geqslant 0$.
 Choose $\tau=\frac{\epsilon}{4}$. 
 	Then we obtain $\omega=\beta_2+d\eta_2$, where \begin{equation*}
 		\beta_2=\tilde\beta^{1,1}-\partial\zeta-\pp\olsi{\zeta}
 		, \ \ \ \ \eta_2=\zeta+\olsi{\zeta}+\tilde\eta,
 	\end{equation*} and $|\nabla_{\omega}^k\beta_2|_{\omega}=O(r^{-2+\frac{\epsilon}{2}-k})$ for all $k\geqslant 0$.  Notice $\beta_2$ and $d\eta_2$ are both real-valued $(1,1)$ forms.
 	
 	\

  	\textbf{Step 3.}  Let $\chi$ be a smooth compactly supported function on $X$ which equals 1 in a neighborhood of $X_R$. Then  $\eta_2=\chi\eta_2+\eta_2'$, where $\eta_2'$ vanishes in a neighborhood of $X_R$. Writing
  		$d\eta_2'=dr\wedge\alpha_1+\alpha_2$, where $\p_r\lrcorner \alpha_1=\p_r\lrcorner \alpha_2=0$, we define $\lambda=\int_{R}^{r}\alpha_1 dr$. Then $\lambda$ vanishes in a neighborhood of $X_R$ so extends to a 1-form on $X$. Moreover, one can check that $d\eta_2'=d\lambda$.  The decay of $\beta_2$ implies $|d\eta_2|_{\omega}\leqslant C$ on $X$, so we have $|\lambda|=O(r^{1+\delta})$ for all $\delta>0$. It follows that $d\eta_2=d\gamma$, where $\gamma=\chi\eta_2+\lambda$, with $|\gamma|=O(r^{1+\delta})$ for all $\delta>0$.
  		
  		\
  		
  		\textbf{Step 4.} Since $d\gamma=\omega-\beta_2$ is of type $(1,1)$, we have $\bp\gamma^{0, 1}=0$.  Since $H^1(X_R,\mathcal{O})=0$ by \eqref{vanishing of H^1}, there exists a smooth function $f$ on $X_R$ such that $\pp f=\gamma^{0,1}$ on $X_R$. Choose a cut off function $\chi_0$ which equals 1 in a neighborhood of the exceptional set of $\pi$ and 0 outside $X_{R-\epsilon'}$ for some small positive number $\epsilon'$. 
  		We now apply Theorem \ref{Hormander} (with $\varphi_M=(n+2+\frac \epsilon 2)\varphi$) to obtain  a function $u$ satisfying $\bp u=\gamma^{0,1}-\pp(\chi_0f)$ and 
 	\begin{equation*}
 		\int_X |u|^2e^{-(n+2+\frac{\epsilon}{2})\varphi}\omega^n<\infty.
 	\end{equation*}
Then $d\gamma=\sqrt{-1}\p\bp\psi_1$, where $\psi_1=\sqrt{-1}(u+\chi_0f-\bar u-\bar\chi_0\bar f)$. Again by Proposition \ref{c:elliptic regularity} we see
 	 $|\nabla_\omega^k\psi_1|_{\omega}=O(r^{2+\epsilon-k})$ for all $k\geqslant 0$.

\

\textbf{Step 5.} We have achieved that $\omega=\beta_2+\sqrt{-1}\p\bp\psi_1$ with estimates on both $\beta_2$ and $\psi_1$. It remains to modify $\beta_2$ to be a K\"ahler form. By the asymptotics for the complex structure and the Riemannian metric obtained in Lemma \ref{l:asymptotic decay}, we know that there exists a large constant $\olsi{R}>R$ such that on the region $\{r>\olsi{R}\}$ we have	$\ii\partial\pp r^{\epsilon}\geqslant c r^{-2+\epsilon}\omega$ for some positive number $c>0$. 
Choose a smooth non-decreasing convex function $u:\mathbb R\rightarrow \mathbb R$ such that $u(t)=1$ for $t\leqslant 2\olsi{R}^{\epsilon}$ and $u(t)=t$ for $t\gg1$ and let $h=u\circ r^{\epsilon}$. Then $h$ can be naturally viewed as a smooth plurisubharmonic function on $X$ and outside a compact set $K'$ we have 
\begin{equation*}
	\ii\partial\pp h\geqslant c r^{-2+\epsilon}\omega
\end{equation*}
for some positive number $c>0$. Choose a smooth compactly supported function $\chi$ on $X$ which equals 1 on a neighborhood of $K'$. Then 
 we define $$\beta=\beta_2+\ii\partial\pp(\chi\psi_1+Ah),$$
 $$\psi=(1-\chi)\psi_1-Ah, $$
 where $A$ is a constant. One can choose $A$ large  so that $\beta$ is K\"ahler. 
 Then $\omega=\beta+\sqrt{-1}\p\bp\psi$. It follows from the above construction that we have the desired estimates. 
\end{proof}

 \subsection{Construction  of background K\"ahler forms}\label{section:defining projection map}  Recall some notations in Section \ref{two step degeneration}:  $X'\subset \mathbb C^N$ is the affine variety defined by the ring $R(X)$ of holomorphic functions on X with polynomial growth, $W$ is the weighted asymptotic cone of $X'$ with respect to $\Lambda$ and $Y_j=\Lambda^{j-1}.X'$. So in particular away from $0$ we know that $Y_j$ converges to $W$ as smooth submanifolds in $\mathbb C^N$. For notational convenience we also denote $Y_\infty\equiv W$.   We denote by $\text{Sing}(Y_j)$ the singular set of $Y_j$, which  consists of finitely many points, and denote $Y_j^\circ=Y_j\setminus \text{Sing}(Y_j)$.    We also define $\mathfrak S\equiv\cup_{1\leqslant j\leqslant \infty}\text{Sing}(Y_j)$.

\begin{proposition}\label{construction of projection maps}
   There exist a constant $\delta_0>0$, a compact set $K\subset \mathbb C^N$ which contains $\mathfrak S$, and  for each $j\in \mathbb Z_{\geqslant 1}\cup\{\infty\}$ a  diffeomorphism $p_j:X'\setminus K\rightarrow Y_j\setminus K$  such that for all $k\geqslant 0$  
   \begin{equation*}
   |\nabla_{\omega_\xi}^k (p_j-\Id)|_{\omega_\xi}\leqslant C_kr_\xi^{-\delta_0-k},
   \end{equation*}
where we view both $p_j$ and $\Id$ as maps into $\C^N$, and    \begin{equation*} \label{e: complex strucutre close}
   	\left|\nabla^k_{\omega_\xi}\left(p_j^*(J_{Y_j})-J_{X'}\right)\right|_{\omega_\xi}\leqslant C_kr_\xi^{-\delta_0-k}.   \end{equation*} 
\end{proposition}
\begin{remark}
Here the definite rate $\delta_0$ arises from the algebraicity of $W$. It is of crucial importance in our argument as it beats all the (arbitrarily small polynomial order) error when we compare $X$ to the asymptotic cone $\mathcal C$ using the 2-step degeneration theory (Proposition \ref{rough estimate}).
\end{remark}
\begin{proof}

By definition  we can find finitely many holomorphic polynomials $F_1, \cdots, F_m$ on $\C^N$ which generate the ideal of $X'$, such that $G_1, \cdots, G_m$ generate the ideal of $W$. Here $G_\alpha$ is the leading term of $F_\alpha$ with respect to the $\Lambda$ action. Then we know that for all $l\in \mathbb N$ we have $$F_\alpha(\Lambda^{-l}. z)=\lambda^{-l c_\alpha} (G_\alpha(z)+\lambda^{l \delta_\alpha}E_{\alpha l}(z))$$ for some constants  $\delta_{\alpha}, c_\alpha>0$ and some polynomials $E_{\alpha l}$ whose coefficients are uniformly bounded independent of $l$. Denote $F_{\alpha, l}= (\Lambda^{-l})^*F_\alpha$, then $F_{\alpha, 1}\cdots, F_{\alpha, l}$ generate the ideal of $Y_l$. Let $\delta_0=\min_\alpha \delta_\alpha. $

We define $p_\infty$ to be the normal projection map from $X'$ to $W$, i.e., for $x\in X'$, we let $p_\infty(x)$ to the unique point in $W$ that is closest to $X'$ with respect to the metric $\omega_\xi$. We need to show this is well-defined outside a compact set in $X$ and satisfies the desired properties. 
For $l\geqslant 1$  we denote by $A_l$ the annulus in $\C^N$ defined by $\lambda^{-l+1}\leqslant r_\xi\leqslant \lambda^{-l-1}$. Then $\Lambda^l$ maps $X'\cap A_l$ to $Y_l\cap A_1$. By the conical nature of $\omega_\xi$, it suffices to consider the normal projection from $Y_l\cap A_1$ to $W$ for $l\gg1$.  

\begin{claim}
  there is a covering of $W\cap A_1$ by finitely many open sets of the form $U_\gamma\cap W$ where $U_\gamma\subset\C^N$ such that for $l$ large there are holomorphic functions $h_{\gamma,l}: U_\gamma\cap W\rightarrow\C^N$ with $|\nabla^k_{\omega_\xi}(h_{\gamma, l}-\Id)|\leqslant C_k \lambda^{l\delta_0}$ for all $k\geqslant 0$ and  $Y_l\cap A_1\subset \cup_\gamma \text{Im}(h_{\gamma,l})$.
\end{claim}

To prove the \textbf{Claim}, we fix a point $w\in W\cap A_1$. Then we can find a neighborhood $U\subset \C^N$, such that $W\cap U$ is given by the zero set of $N-n$ number of  $G_\alpha$'s (in the scheme-theoretic sense). For simplicity of notation we may assume these are $G_1, \cdots G_{N-n}$. Shrinking $U$ if necessary we may find local holomorphic coordinates $\{\zeta_1, \cdots, \zeta_N\}$ such that $\zeta_\alpha=G_\alpha$ for $\alpha=1, \cdots N-n$. Now we have  $F_{\alpha, l}=\lambda^{-lc_\alpha}(G_\alpha(z)+\lambda^{l\delta_\alpha} E_{\alpha l}(z))$.  It follows that for $l$ large the common zero set of $F_{\alpha, l}(\alpha=1, \cdots, N-n)$ is a smooth complex submanifold in $U$, so in particular it agrees with $Y_l\cap U$. Using the local  coordinates $\{\zeta_1, \cdots, \zeta_N\}$ it is easy to see that $Y_l\cap U$ is contained in the image of a holomorphic function $h_l:U\rightarrow \C^N$ such that $|\nabla^k_{\omega_\xi}(h_l-\Id)|_{\omega_\xi}\leqslant C_k\lambda^{l\delta_0}$ for all $k\geqslant 0$. Since $A_1$ is compact the \textbf{Claim} follows.

 \

Since the normal injectivity radius of $W\cap A_1$ is uniformly bounded, there is a tubular neighborhood $\mathcal N$ of $W\cap A_1$ such that the normal projection map $\Pi: \mathcal N\rightarrow W$ is smooth.  It follows from the \textbf{Claim} that for $l$ large, $Y_l\cap A_1\subset \mathcal N$.  So $p_\infty$ is well-defined and smooth outside a compact set in $X'$. It is straightforward to check that it satisfies the desired derivative bounds.   The estimates on the complex structures follow from the fact that $h_{\gamma, l}$ is holomorphic.

Now for each $j<\infty$ we can perform the above arguments uniformly (noticing that the closeness between $Y_j$ and $W$ gets improved as $j$ gets larger) to construct the corresponding projection maps $q_j: Y_j\setminus K\rightarrow W$ with uniform estimates. Then we let $p_j=q_j^{-1}\circ p_\infty$.  

\end{proof}

Now we fix $\epsilon_0\ll \delta_0$ and apply Theorem \ref{decomposition of kahler forms} with $\epsilon=\epsilon_0$ to get a decomposition on $X$
\begin{equation}\label{e:decomposition}
\omega=\beta+\sqrt{-1}\p\bp\psi.
\end{equation} Let $\psi_j'=\psi\circ p_j^{-1}$ and $r_j=r\circ p_j^{-1}.$  Denote $\gamma_j\equiv\sqrt{-1}\p\bp\psi_j'$.
By Lemma \ref{l:asymptotic decay},  Proposition \ref{rough estimate}, Corollary \ref{estimates for different covariant derivatives} and Proposition \ref{construction of projection maps}, we have
\begin{corollary}\label{c:metric close}
There exists a compact set $K\subset \C^N$ such that for all $j\in \mathbb Z_{\geqslant1}\cup\{\infty\}$, $\gamma_j$ is a K\"ahler form on $Y_j\setminus K$. Moreover,
\begin{itemize}
\item[(1).] for all $k\geqslant 0$ and $\delta>0$,
\begin{equation*}C_\delta^{-1}r_j^{1-\delta}\leqslant r_\xi\leqslant C_\delta r_j^{1+\delta} \text{ and }	|\nabla_{\gamma_j}^k(\gamma_j-\omega_\xi)|_{\gamma_j}\leqslant C_{\delta,k} r_j^{\delta-k},
\end{equation*}

	\item[(2).] for all $k\geqslant 0$ and $\delta<\delta_0-\epsilon_0$,  \begin{equation} \label{e:pull back metric close}
	|\nabla_\omega^k(p_j^*\gamma_j-\omega)|_{\omega}\leqslant C_{\delta, k} r_j^{-\delta-k}.
	\end{equation}
	\end{itemize}
\end{corollary}

Next we  extend $\gamma_j$ to a global K\"ahler form on $Y_j^\circ$ in a uniformly controlled way. In the following, we denote by  $\omega_0=\frac{\ii}{2}\partial\pp r_0^2$ the K\"ahler form corresponding to the standard Euclidean metric on $\mathbb C^N$, where $r_0$ is the radius function for the Euclidean metric.

\begin{lemma}\label{construction of exact background kahler forms}
	There exist compact sets $K_0\subset K_1\subset \mathbb C^N$ containing a neighborhood of $\mathfrak S$, a constant $A_0$ and K\"ahler forms $\omega_j=\ii\partial\pp \psi_j$ on $Y_j^{\circ}$, such that for all $j\in \mathbb Z_{\geqslant1}\cup\{\infty\}$,
	\begin{itemize}
	    \item[(1).] $\omega_j=A_0\omega_0$ on $K_0\cap Y_j^\circ$,
	    \item[(2).] $\psi_j=\psi_j'$ on $Y_j\setminus K_1$.
	\end{itemize}

	\end{lemma}
\begin{proof}
	By Proposition \ref{rough estimate} and Proposition \ref{construction of projection maps}, we know that there exist constants $c, T_1>0$ such that 
$\ii\partial\pp \psi_j'\geqslant cr_\xi^{-\frac{1}{2}}\omega_{\xi}$
	on $\{r_{\xi}\geqslant T_1/2\}\cap Y_j$. 
Using a cut-off function on $\C^N$  we may extend   $\psi_j'$ smoothly to $Y_j^\circ$,  so that it vanishes when $r_{\xi}\leqslant T_1/2$, equals $\psi\circ p_j^{-1}$ when $r_{\xi}\geqslant T_1$ and satisfies
$	\ii\partial\pp \psi_j'\geqslant -\frac{A}{2}\omega_{0} 
$ on  $\{r_{\xi}\leqslant T_1\}\cap Y_j$ for some constant $A>0$ independent of $j$.

Choose a smooth increasing and convex function $u: \mathbb{R} \rightarrow \mathbb{R}$ such that $u(t)=T_1^2 / 3$ for $t \leqslant T_1^2 / 2$ and $u(t)=t$ for $t \geqslant T_1^2$. Define the function  $r_{1}^2= u \circ r^2_{\xi}$ on $\C^N$. Then $\ii\partial\pp r_1^2 $ is globally non-negative  and equals $\omega_{\xi}$ on $\{r_{\xi}\geqslant T_1\}$.  Choose a cutoff function $\rho:\mathbb R\rightarrow\mathbb R$ which equals 1 when ${t\leqslant 2T_1^2}$ and equals $0$ when ${t\geqslant 4T_1^2}$. Then it is straightforward to check that there exists a small positive constant $\epsilon$ such that $\phi=\epsilon(\rho \circ r_{\xi}^2)\cdot r_0^2+r_1^2$ satisfies that $\sqrt{-1}\p\bp\phi$ is a K\"ahler form on $\C^N$. By construction there exists a constant $T>1$ such that $\{r_{\xi}\leqslant 2T_1\}\subseteq \{\phi\leqslant T\}$.

It is easy to find a smooth function $\chi:[0,\infty)\rightarrow\mathbb R$ such that the following hold
	\begin{itemize}
\item	$\chi(t)=A_0t$ if $t\leqslant T$, where $A_0$ is a constant bigger than $A\epsilon^{-1}$;
\item $\chi(t)=0$ for $t\gg1$;
\item for all $t>0$ we have $\chi''\geqslant-\frac{c}{8}t^{-\frac{3}{2}}$ and $\chi'\geqslant -\frac{c}{8}t^{-\frac{1}{2}}$. 
	\end{itemize}Then $\psi_j=\chi\circ \phi+\psi_j'$ satisfies the desired properties. 
	\end{proof}
	
	Recall that we have the crepant resolution map $\pi: X\rightarrow X'$, thus for $j\in \mathbb Z_{\geqslant 1}$ the composition map $\Lambda^j\circ \pi$ is a crepant resolution of $Y_j$. We denote the latter by $\pi_j:\hat Y_j=X\rightarrow Y_j$. For all $j\in \mathbb Z_{\geqslant 1}$ and $\epsilon\in (0,\lambda^{4j})$ (we emphasize here that the domain of $\epsilon$ depends on $j$), we have the following family of K\"ahler forms on $\hat Y_j$
$$\widetilde{\omega}_{j,\epsilon}\equiv \epsilon \beta +\pi_j^*\omega_j=\epsilon\beta+\ii\partial\pp \pi_j^*\psi_j,$$ where $\beta$ is given in \eqref{e:decomposition}. Notice that $\beta$ and $\widetilde{\omega}_{j,\epsilon}$ can also be naturally viewed as forms on $Y_j^\circ$ via $(\pi_j^{-1})^*$. The following lemma (together with Corollary \ref{c:metric close}-(2)) shows that they are uniformly asymptotic to the cone $\mathcal C$.
\begin{lemma}\label{decay of beta}
    There exists a compact set $K\subset \mathbb C^N$ such that for all $j\in \mathbb Z_{\geqslant 1}$ and $k\geqslant0$ on $Y_j\setminus K$  we have
    \begin{equation*}
        \lambda^{4j}\left|\nabla^k_{\omega_j}(\pi_j^{-1})^*\beta\right|_{\omega_j}\leqslant C_k r_j^{-\frac{1}{2}-k}.
    \end{equation*}
\end{lemma}

\begin{proof}
Outside a compact set, we identify $X$ with $X'$ and therefore we may just view $\pi_j$ as $\Lambda^j$. Let us first prove the case $k=0$. By Corollary \ref{c:metric close} and the fact $\Lambda^*\omega_\xi=\lambda^{2}\omega_\xi$, we know that for any $\epsilon_1,\epsilon_2>0$, there exists constants $C_{\epsilon_1}$ and $C_{\epsilon_2}$ such that
\begin{equation}\label{replacement of omega and omega_xi}
\begin{aligned}
 \left|(\Lambda^{-j})^*\beta\right|_{\omega_j}&\leqslant C_{\epsilon_1}\left|(\Lambda^{-j})^*\beta\right|_{\omega_{\xi}} r_{\xi}^{\epsilon_1}\leqslant C_{\epsilon_1}(\Lambda^{-j})^*(\left|\beta\right|_{(\Lambda^j)^*\omega_{\xi}}) r_{\xi}^{\epsilon_1}\\
 &\leqslant C_{\epsilon_1}\lambda^{-2j} (\Lambda^{-j})^*(\left|\beta\right|_{\omega_{\xi}})r_{\xi}^{\epsilon_1}\leqslant C_{\epsilon_1}C_{\epsilon_2}\lambda^{-2j} (\Lambda^{-j})^*(\left|\beta\right|_{\omega}r_{\xi}^{\epsilon_2})r_{\xi}^{\epsilon_1} \\
 &\leqslant C_{\epsilon_1}C_{\epsilon_2}\lambda^{-2j-\epsilon_2j}(\Lambda^{-j})^*(r^{-1}) r_{\xi}^{\epsilon_1+\epsilon_2},
\end{aligned}
\end{equation}where for the last inequality we used the estimate of $\beta$ with respect to the metric $\omega$ obtained in Theorem \ref{decomposition of kahler forms}. 
By choosing $\epsilon_1$ and $\epsilon_2$ small enough, one can directly show that there exists a positive constant $C_0$ such that
\begin{equation*}
     \lambda^{4j}\left|(\Lambda^{-j})^*\beta\right|_{\omega_j}\leqslant C_0 r_j^{-\frac{1}{2}}.
\end{equation*}
This proves the case $k=0$. Note that suppose we have two Riemannian metrics $g_1$ and $g_2$ then for any tensor $\eta$ and another fixed background metric $g_0$, we have
\begin{equation}\label{covariant derivative with respect to different metrics}
\begin{aligned}
 \left|\nabla_{g_1}^k\eta\right|_{g_0}\leqslant& C_k\sum_{\substack{k_1+k_2=k\\k_i\geqslant0}} \left|\nabla_{g_1}-\nabla_{g_2}\right|_{g_0}^{k_1}\left|\nabla_{g_2}^{k_2}\eta\right|_{g_0}\\
 &+C_k\sum_{\substack{k_1+k_2+k_3=k-1\\k_i\geqslant0, k_1\geqslant1}} \left|\nabla_{g_2}^{k_1}(\nabla_{g_1}-\nabla_{g_2})\right|_{g_0}\left|\nabla_{g_1}-\nabla_{g_2}\right|_{g_0}^{k_2}\left|\nabla_{g_2}^{k_3}\eta\right|_{g_0},
\end{aligned}
\end{equation} where we view $\nabla_{g_1}-\nabla_{g_2}$ as a tensor and  make a convention that $\nabla^0_{g_2}\eta=\eta$.
Then the proof for $k\geqslant 1$ follows from the same argument as we did in \eqref{replacement of omega and omega_xi} by using \eqref{covariant derivative with respect to different metrics} together with Corollary \ref{estimates for different covariant derivatives} and Theorem \ref{decomposition of kahler forms}.
\end{proof}

	Recall that $\Omega$ is the given holomorphic volume form on $X$. Since $\Omega$ is holomorphic and satisfies the equation $\omega^n=(\ii)^{n^2}\Omega\wedge \ols{\Omega}$, we obtain $\nabla_{\omega}\Omega=0$. We can view $\Omega$ naturally as a holomorphic volume form on $X'=Y_1$ (notice the canonical bundle of each $Y_j$ is trivial).
	
\begin{lemma}\label{construction of holomorphic volume form}
For any $j\in \mathbb Z_{\geqslant 1}$,	there exists a  nowhere vanishing holomorphic volume form $\Omega_j$ on $Y_j$ satisfying:
	\begin{itemize}
		\item [(1).]there exist a constant $\delta_1>0$  and a compact set $K\subset \C^N$ such that on $Y_j\setminus K$, for all $k\geqslant 0$  the function $h_j=\log\left(\frac{\Omega_j\wedge\ols{\Omega}_j}{(p_j^{-1})^*(\Omega\wedge\ols{\Omega})}\right)$ satisfies  
	\begin{equation*}
		|\nabla^k_{\omega_j}h_j|\leqslant C_k r_j^{-\delta_1-k}.
	\end{equation*} 
	\item[(2).] for every compact set $K_1\subset \mathbb C^N$ we have
	\begin{equation*}
		\sup_{j\in \mathbb Z_{\geqslant1}}\int_{K_1\cap Y_j}(\ii)^{n^2}\Omega_j\wedge\ols{\Omega}_j<\infty.
	\end{equation*}

	\end{itemize}
	\end{lemma}
	\begin{remark}
    It follows from the proof that for $j\in \mathbb Z_{\geqslant 1}$, $\Omega_j$ coincides with $(\Lambda^{-j})^*\Omega$ up to a nonzero multiplicative constant depending on $j$. For our purpose we need the precise asymptotics of $\Omega_j$ and the stated properties of $h_j$ essentially determine this constant. Moreover from the proof, we also obtain 
    \begin{equation}\label{fix angle}
        \left|\nabla_{\omega_j}^k((p_j^{-1})^*\Omega-\Omega_j)\right|_{\omega_j}\leq C_k r_j^{-\delta_1-k} \text{ on $Y_j\setminus K$, for all $k\geqslant 0$}.
    \end{equation} Note that (1) and (2) in the proposition are invariant if we multiply $\Omega_j$ by a complex number with norm 1 and the estimate \eqref{fix angle} can be used to eliminate this freedom.
\end{remark}
\begin{proof}
Let  $\varphi  =\log(r_{\xi}^2+1)-\frac{1}{100}(\log(r_{\xi}^2+3))^{\frac{1}{2}}$ and denote $\varphi_j=\varphi|_{Y_j}$. Then by Lemma \ref{psh function on a kahler cone}-(1) and Corollary \ref{c:metric close}, we know that outside a compact set $K$ for every $\epsilon>0$, there is a positive constant $c_{\epsilon}$ such that 
$
    \ii\p\pp \varphi_j\geqslant c_{\epsilon}r_j^{-2-\epsilon}\omega_j.
$
Let $p_j$ be given in Proposition \ref{construction of projection maps}. For all $\delta\in (0, \delta_0-\epsilon_0)$ we have 
			\begin{equation} \label{e:almost constant}
		    \left|\left((p_j^{-1})^*\Omega\right)^{n,0}\right|_{\omega_j}-1=O(r_j^{-\delta}),
		\end{equation}
	and for all $k\geqslant 1$, 
	$
		\left|\nabla^k_{\omega_j}\left((p_j^{-1})^*\Omega\right)^{n,0}\right|_{\omega_j}\leqslant C_{\delta, k} r_j^{-\delta-k}
	$ since $\nabla_{\omega}\Omega=0$. 
	
Fix $T$ large so that $K\subset \{r_\xi\leqslant T\}$, and fix  a cut off function $\chi$ on $\C^N$ which equals 0 when $r_\xi\leqslant T$, equals 1 when $r_\xi>2T$. Recall that we use $Y_j^{\circ}=Y_j\setminus \text{Sing}(Y_j)$ to denote the smooth locus of $Y_j$. Since $Y_j$ is affine and $ \text{Sing}(Y_j)$ consists of finitely many points, we know $Y_j^\circ$ admits a complete K\"ahler metric. Then for any $\tau>0$ small, we can apply Theorem \ref{Hormander} on $M=Y_j^{\circ}$, with $\omega_M=\omega_j$ and $\varphi_M=(n-1+\tau)\varphi_j$, to obtain an $(n,0)$ form $v_j$  satisfying
	$\label{partial bar equation 2}
		\pp v_j=\pp \left(\chi\cdot  ((p_j^{-1})^*\Omega)^{n,0}\right), 
	$ such that
	$$\int_{Y_j^{\circ}}e^{-(n-1+\tau)\varphi_j}(\ii)^{n^2}v_j\wedge \ols{v}_j\leqslant C. $$
 By Proposition \ref{c:elliptic regularity}, when $r_\xi\geqslant 2T$, we have $|\nabla^k_{\omega_j}v_j|\leqslant C_k r_j^{-\delta_1-k}$ for some constant $\delta_1>0$ and all $k\geqslant0$. 
	
  Now let $\Omega_j=\chi ((p_j^{-1})^*\Omega)^{n,0}-v_j$. It is holomorphic on $Y_j^{\circ}$ and is non-zero by \eqref{e:almost constant} and the decay of $v_j$. Notice that the function $\frac{(\Lambda^{j})^*\Omega_j}{\Omega}$ is a  holomorphic function on $X$ which is $O(r^{\epsilon})$ for all $\epsilon>0$, so must be a nonzero constant. In particular, $\Omega_j$ is a nowhere vanishing holomorphic volume form on $Y_j$. The desired estimates follow from straightforward computation.
\end{proof}

In the following proposition, we omit the notation $(\pi_j^{-1})^*$ and view $\widetilde{\omega}_{j,\epsilon}$ as a K\"ahler form on $Y_j^\circ \subset \mathbb C^N$. We also view $r_j$ (which is a priori only defined on the end of $Y_j$ using $r\circ p_j^{-1}$) as a positive smooth function on $Y_j$ which is a constant on $K_0\cap Y_j$ where $K_0\subset \mathbb C^N$ is a fixed compact set independent of $j$, containing $\mathfrak S=\cup_{1\leqslant j\leqslant \infty}\text{Sing}(Y_j)$.
\begin{proposition}\label{preconditioning}
There exist a $\delta_2>0$, a compact set $K\subset \C^N$ containing a neighborhood of $\mathfrak S$ and for all $j\in\mathbb Z_{\geqslant 1}$ and $\epsilon\in (0, \lambda^{4j})$ a smooth function  $u_{j,\epsilon}$ on $Y_j^\circ$ satisfying
\begin{itemize}
    \item[(1).] for all $k\geq0$, 
	$\left|\nabla^k_{\widetilde\omega_{j,\epsilon}}u_{j,\epsilon}\right|_{\widetilde\omega_{j,\epsilon}}\leqslant C_k r_j^{-\delta_2+2-k}$ on $Y_j\setminus K$,
    \item[(2).] let $e^{-f_{j,\epsilon}}=\frac{\left(\widetilde\omega_{j,\epsilon}+\ii\partial\pp u_{j,\epsilon}\right)^n}{(\ii)^{n^2}\Omega_j\wedge\ols{\Omega}_j}$, then on $Y_j\setminus K$ for all $k\geq0$,  $
	\left |\nabla^k_{\widetilde\omega_{j,\epsilon}}f_{j,\epsilon}\right|_{\widetilde\omega_{j,\epsilon}}\leqslant C_k r_j^{-\delta_2-2-k}$,
	 \item[(3).] $u_{j,\epsilon}$ is supported on $Y_j\setminus K$, 
    \item [(4).]
        $\omega_{j}+\sqrt{-1}\p\bp u_{j,\epsilon}\geqslant \frac{1}{2}\omega_j$ and for all $k\geqslant 0$, $\left|\nabla^k_{\omega_j}u_{j,\epsilon}\right|_{\omega_j}\leq C_k (r_j+1)^{-\delta_2+2-k}$ on $Y_j^\circ$.
	
\end{itemize}
\end{proposition}

\begin{proof}
We first identify the end of $X'$ differentiably with $\mathcal Q_{R_0}$ via Lemma \ref{l:asymptotic decay}. Then using the maps $p_j$ constructed in Proposition \ref{construction of projection maps} we may identify the end of all $Y_j$ with $\mathcal Q_{R_1}$ for some $R_1>R_0$. In particular we may view the metrics $\widetilde g_{j,\epsilon}$ (associated to $\widetilde\omega_{j,\epsilon}$) on $Y_j^\circ$ as a family of Riemannian metrics on $\mathcal Q_{R_1}$. By Lemma \ref{l:asymptotic decay}, Corollary \ref{c:metric close} and Lemma \ref{decay of beta} we know that  for all $j$,  $$s^l\sup_{\p\mathcal Q_s}|\nabla^l_{g_\mathcal C}(\widetilde g_{j,\epsilon}-g_\mathcal C)|_{g_\mathcal C}\leqslant e_l(s)$$ for some $e_l(s)$ with $\lim_{s\rightarrow\infty} e_l(s)=0$. Then we can apply Proposition \ref{p:Poisson equation general}  to solve the Poisson equation with uniform estimates for the metric $\widetilde g_{j, \epsilon}$ on $\mathcal Q_{R_2}$, where $R_2\geqslant R_1$ is given in Proposition \ref{p:Poisson equation general}. More precisely, we solve the equation in a fixed H\"older space for $k$ large and then apply elliptic regularity as Proposition \ref{c:elliptic regularity} to gain the decay for all derivatives. Next we can argue as in \cite[Section 2.4]{CH1} to obtain functions $u_{j,\epsilon}$ defined only on the end of $Y_j$ satisfying item (1) and (2) in the proposition. Fix a smooth function $H(t):[\,0,\infty)\rightarrow [\,0,1\,]$, which equals 0 for $t\leqslant 1$ and equals 1 for $t\geqslant 2$. Combining the estimate in (1) with Lemma \ref{construction of exact background kahler forms} and Lemma \ref{decay of beta}, it is direct to show that we can choose $R_3$ sufficiently large independent of $j$ and $\epsilon$, such that  $H(\tfrac{r_j}{R_3})u_{j,\epsilon}$ is a global defined function on $Y_j$ satisfying all (1)-(4) in the proposition.
\end{proof}

For all $j\in \mathbb Z_{\geqslant 1}$ and $\epsilon\in (0,\lambda^{4j})$, we define K\"ahler forms on $\hat Y_j$ $$\omega_{j, \epsilon}\equiv \widetilde{\omega}_{j,\epsilon}+\ii\partial\pp \pi_j^*u_{j, \epsilon},$$ and we denote by $g_{j, \epsilon}$ and $J_{j, \epsilon}$  the corresponding Riemannian metric and  complex structure. It follows from Lemma \ref{l:asymptotic decay}, Proposition \ref{construction of projection maps}, Corollary \ref{c:metric close}, Lemma \ref{decay of beta} and Proposition \ref{preconditioning} that $(g_{j, \epsilon},J_{j, \epsilon})$ are still uniformly asymptotic the cone $\mathcal{C}$. More precisely, we have
\begin{proposition}\label{close to asymptotic cone C}
There exist a compact set $K\subset \mathbb C^N$ and a large positive number $R_2$ such that  for all $j\in \mathbb Z_{\geqslant 1}$, the diffeomorphism $P_j=p_j\circ \Phi:\mathcal Q_{R_2}\rightarrow Y_j\setminus K$ satisfies that for all $l\geqslant 0$,
\begin{equation*}
\lim_{s\rightarrow\infty}\sup_{j\geqslant 1}\sup_{\p \mathcal Q_s}s^l\left (|\nabla^l_{g_\mathcal C}(P_j^*g_{j, \epsilon}-g_{\mathcal C})|_{g_\mathcal C}+|\nabla^l_{g_\mathcal C}(P_j^*J_{j, \epsilon}-J_{\mathcal C})|_{g_\mathcal C}\right)=0,
\end{equation*} where we omit $(\pi_j^{-1})^*$ again and view $g_{j, \epsilon}$ and $J_{j, \epsilon}$ as tensors on $Y_j$.
\end{proposition}

\section{Uniform estimates for the complex Monge-Amp\`ere equation}\label{s5}
From now on for $j\in \mathbb Z_{\geqslant 1}$  we will abuse notations and identify $\omega_j$, $\Omega_j$ and $r_j$ with their pull-backs under $\pi_j:\hat Y_j\rightarrow Y_j$. We now look for Calabi-Yau metrics on $\hat Y_j$ which are asymptotic to $\omega_{j, \epsilon}$. To do this  we need to solve on $\hat Y_j$ the equation
\begin{equation}\label{complex monge-ampere equation}
\begin{aligned}
	\left(\omega_{j, \epsilon} +\ii\partial\pp \varphi_{j, \epsilon}\right)^n=(\ii)^{n^2}\Omega_j\wedge\ols{\Omega}_j.
\end{aligned}
	\end{equation}
	 Let $\mu_j$ denote the volume form $(\ii)^{n^2}\Omega_j\wedge\ols{\Omega}_j$ which is also equal to $e^{f_{j,\epsilon}}\omega_{j,\epsilon}^n$. It is proved in Proposition \ref{preconditioning} that
	there exist a $\delta_2>0$ and a compact set $K\subset \mathbb C^N$ such that on $\hat Y_j\setminus \pi_j^{-1}(K\cap Y_j)$ for all $k\geqslant0$,
		\begin{equation}\label{decaying of Ricci potential}
			\left|\nabla^k_{\omega_{j, \epsilon}}f_{j, \epsilon}\right|_{\omega_{j, \epsilon}}\leqslant C_k r_{j}^{-\delta_2-2-k}.
		\end{equation} 
 The following theorem can be derived from Tian-Yau \cite[Section 4]{TY2}. See also Hein \cite[Section 4]{hein1}.
 
 \begin{theorem}[{\cite{TY2,hein1}}]\label{ existence result of Tian-Yau}
 	For fixed $j$ and $\epsilon$, (\ref{complex monge-ampere equation}) admits a smooth solution $\varphi_{j, \epsilon}$ on $\hat Y_j$ such that 
 	\begin{itemize}
 		\item [(1)] there exists a constant $p_0$ depending only on $\delta_2$ such that $\Vert\varphi_{j, \epsilon}\Vert_{L^{p_0}(\mu_j)}<\infty$ and $\varphi_{j, \epsilon}(x)\rightarrow 0$ as $ r_j(x)\rightarrow \infty$.
 		\item [(2)] $\left\Vert\nabla^k_{\omega_{j, \epsilon}}\varphi_{j, \epsilon}\right\Vert_{L^{\infty}(\omega_{j,\epsilon})} <\infty$ for all $k\geqslant0$.
 	\end{itemize}
 \end{theorem}
 \begin{remark}
 	The proof of the $L^{\infty}$ bound of $\varphi_{j, \epsilon}$ is via the Moser iteration, thus indeed we know that $\left \Vert\varphi_{j, \epsilon}\right\Vert_{L^p(\mu_j)}<\infty$ for all $p\gg1$. In the following, we fix a large  $p_0$ such that $q_0=\frac{n(p_0+1)}{n+p_0}$ satisfies 
 	\begin{equation}\label{assumption on index q_0}
 		(2+\delta_2)q_0>2n.
 	\end{equation}
 \end{remark}
 
 Since the background metrics $\omega_{j,\epsilon}$ are asymptotic to  the cone metric on $\mathcal{C}$, by the method in \cite[Section 5]{TY2} and using ${r_j}^{-2\delta}$ as a barrier to apply the maximum principle, we can prove the following decay estimates for  $\varphi_{j,\epsilon}$. We remark that   alternatively, one can also use the method of \cite{CH1, CGT} to solve the complex Monge-Amp\`ere equation in a suitable weighted space (since we have Lemma \ref{p:weighted-esimate-quotient-space}),  then the decay property will be a direct by-product.
 
  \begin{lemma}\label{non uniform estimate}
  	There exist a constant $\delta_3>0$ depending only on $\delta_2$, compact sets $K_{j,\epsilon}\subset \mathbb C^N$ depending only on $\left\Vert\varphi_{j,\epsilon}\right\Vert_{L^{\infty}}$  and  constants $Q_{k, j, \epsilon}$  depending only on $k$ and  $\left\Vert\varphi_{j,\epsilon}\right\Vert_{L^{\infty}}$ such that  on $Y_j\setminus K_{j,\epsilon}$ for all $k\geqslant0$, 
  	\begin{equation*}
  		 \left|\nabla^k_{\omega_{j, \epsilon}}\varphi_{j, \epsilon}\right|_{\omega_{j,\epsilon}}\leqslant Q_{k, j, \epsilon} { r_j}^{-\delta_3-k}.  
  	\end{equation*}  
  	 \end{lemma}
  
  \begin{proof}
  	We focus on the $C^0$ decay property since the higher order estimates follow from a standard rescaling argument and elliptic estimates (see \cite[Proposition 5.1]{TY2}).  Fix a positive constant $\delta< \min\left\{\frac{\delta_2}{2},\frac{1}{2}\right\}$ and then we can prove the lemma holds for $\delta_3=2\delta$.
  	\begin{claim}
  	 there is a large constant $R_*$ independent of $j$ and $\epsilon$ such that for every $R\geqslant R_*$, on the region $\{r_j\geqslant R\}$   we have
  \begin{equation*}
  	\left(\omega_{j, \epsilon}+R\ii\partial\pp r_{j}^{-2\delta}\right)^n< e^{f_{j, \epsilon}}\omega_{j, \epsilon}^n\text{ and } \left(\omega_{j, \epsilon}-R\ii\partial\pp r_{j}^{-2\delta}\right)^n> e^{f_{j, \epsilon}}\omega_{j, \epsilon}^n.
  \end{equation*}
  	\end{claim}
  The decay of $\varphi_{j,\epsilon}$ follows from this claim by applying the maximal principle to the region $\{r_j\geqslant R_{j,\epsilon}\}$ with
  	$
  	    	R_{j,\epsilon}= \max\{R_*, \left\Vert\varphi_{j,\epsilon}\right\Vert_{L^{\infty}}^{\frac{1}{1-2\delta}}\}.
 $ 
  Let us prove the \textbf{Claim}. We only prove the first inequality since the second inequality follows from the same argument.
  By \eqref{decaying of Ricci potential} and Proposition \ref{close to asymptotic cone C}, we know that  outside a compact set $K_0$, there exists a constant $C_0>0$ such that
  \begin{equation}\label{supersubharmonic}
  \begin{aligned}
  1-C_0r^{-2-\delta_2}\leqslant e^{f_{j,\epsilon}}&\leqslant 1+C_0r^{-2-\delta_2},\\
       \left|\ii\p\pp r_j^{-2\delta}\right|_{\omega_{j,\epsilon}}&\leqslant C_0r_j^{-2(\delta+1)},\\
       \Delta_{\omega_{j,\epsilon}}r_j^{-2\delta}&\leqslant -\delta r_j^{-2(\delta+1)}.
  \end{aligned}
  \end{equation} The last inequality here is due to the fact that on the cone $\mathcal C$, for the radial function $r$ we have $$\Delta_{\omega_{\mathcal C}}r^{-2\delta}=4\delta(-n+\delta+1)r^{-2(\delta+1)}.$$
  Choose  $R_*$ large such that $K_0\subseteq \{r_j\leqslant R_*\}$ and  by the second inequality in \eqref{supersubharmonic} and the choice of $\delta$ we may also assume that for all $R\geqslant R_*$,  on $\{r_j\geqslant R\}$ we have 
  \begin{equation}\label{Laplacian dominates}
    \frac{\left(\omega_{j, \epsilon}+R\ii\partial\pp r_{j}^{-2\delta}\right)^n}{\omega_{j,\epsilon}^n}\leqslant 1+R\Delta_{\omega_{j,\epsilon}}r_j^{-2\delta}+\frac{\delta}{2}r_j^{-2(\delta+1)}, 
  \end{equation}
and
\begin{equation}\label{choice of delta}
    \left(-R\delta+\frac{\delta}{2}\right)r_j^{-2(\delta+1)}\leqslant -2C_0r_j^{-2-\delta_2}.
\end{equation}
    	Then the claim follows from \eqref{supersubharmonic}, \eqref{Laplacian dominates} and \eqref{choice of delta}.
  \end{proof}
 
 Denote   $$\hat\omega_{j, \epsilon}=\omega_{j, \epsilon}+\ii\partial\pp \varphi_{j, \epsilon}.$$ 
Below we will derive uniform estimates on $\varphi_{j, \epsilon}$. Then we will take a double limit,  first as $\epsilon\rightarrow0$ and then as $j\rightarrow\infty$, to obtain a Calabi-Yau metric on $W\setminus\{0\}$. The following arguments are essentially standard, given the large body of literature in this field; see for example the recent work by Collins-Guo-Tong \cite{CGT} which deals with the case when the background metric is asymptotically conical in the stronger sense (with polynomial decay rate).  Notice that we do not need uniform estimates for all $j$ and $\epsilon\in (0, \lambda^{4j})$. Instead we first fix $j$ and derive uniform estimates as $\epsilon\rightarrow 0$, which allows us to take a limit $\varphi_j$. Then we will derive uniform estimates for $\varphi_j$ as $j\rightarrow\infty$. Notice we are free to pass to subsequences when taking limits and we do not need uniqueness of the limits. 
 
 Let $$A_{j, \epsilon}=\sup_{q\geqslant q_0} \Vert  e^{-f_{j, \epsilon}}-1\Vert_{L^q(\mu_j)}, \ \ \ \ A_{j}=\limsup_{\epsilon\rightarrow 0} A_{j,\epsilon}$$ By \eqref{decaying of Ricci potential} and \eqref{assumption on index q_0} we know for each $j,\epsilon$, $A_{j, \epsilon}<\infty$. The following lemma shows that $A_j< \infty$ and moreover we have a uniform bound on $A_j$. 

 \begin{lemma}\label{l:Aj uniform bound}
 There exists a constant $C>0$ such that $A_j\leqslant C$ for all $j$.
 \end{lemma}
 
\begin{proof}
 Let $K$ be the union of the compact sets obtained in Lemma \ref{construction of holomorphic volume form}-(1) and Proposition \ref{preconditioning}, then by (\ref{decaying of Ricci potential}), \eqref{assumption on index q_0} and Lemma \ref{construction of holomorphic volume form}, we obtain there exists a constant $C$ independent of $j$ and $\epsilon$ such that 
 \begin{equation}\label{decomposition of the bound}
     A_{j, \epsilon}\leqslant C\left(\left\Vert e^{-f_{j,\epsilon}}\right\Vert_{L^{\infty}(\pi_j^{-1}(K))}+1\right).
 \end{equation}
By the construction in Lemma \ref{construction of exact background kahler forms} and Proposition \ref{preconditioning}, $\omega_j+\ii \partial\pp u_{j,\epsilon}\leqslant C_1\omega_j\leqslant C_2 \omega_0$ on $Y_j^\circ\cap K$, where $\omega_0$ denotes the K\"ahler form for the Euclidean metric on $\mathbb C^N$, for some constants $C_1$ and $C_2$ independent of $j$ and $\epsilon$. Therefore by the definition of $f_{j,\epsilon}$ we obtain 
\begin{equation}\label{upper bound on compact set}
    \limsup_{\epsilon\rightarrow 0} \left\Vert e^{-f_{j,\epsilon}}\right\Vert_{L^{\infty}(\pi_j^{-1}(K))} \leqslant  C_2^n\left\Vert\frac{\omega_0^n}{(\ii)^{n^2}{\Omega_j}\wedge\ols{\Omega}_j}\right\Vert_{L^{\infty}(K\cap Y_j^\circ)}.
\end{equation}
By the Gauss-Codazzi equation we know the Ricci curvature of $\omega_0|_{Y_j}$ is non-positive, it follows that the function $ \frac{\omega_0^n}{(\ii)^{n^2}{\Omega_j}\wedge\ols{\Omega}_j}$ is plurisubharmonic on $Y_j^\circ=Y_j\setminus \text{Sing}(Y_j)$. Pulling back to $\hat Y_j$ it is a smooth plurisubharmonic function. By the maximal principle and Lemma \ref{construction of holomorphic volume form}, we conclude that the supremum of $ \frac{\omega_0^n}{(\ii)^{n^2}{\Omega_j}\wedge\ols{\Omega}_j}$ on $K\cap Y_j^\circ$ has a uniform upper bound independent of $j$. Then the lemma follows from \eqref{decomposition of the bound} and \eqref{upper bound on compact set}.
\end{proof}
 
As observed by \cite{CGT} (following an earlier idea of Tosatti) it is more convenient to use the Calabi-Yau metrics $\hat\omega_{j, \epsilon}$ as background metrics when we do Moser iteration. By Lemma \ref{non uniform estimate} each $\hat\omega_{j, \epsilon}$ is weakly asymptotic to the cone $\mathcal C$, so by Bishop-Gromov inequality we know the volume of a metric ball of radius $R $ in $(\hat Y_j, \hat \omega_{j, \epsilon})$ is bounded below by $\kappa(\mathcal C) R^{2n}$, where $\kappa(\mathcal C)$ is the volume density of $\mathcal C$. So we have a uniform Sobolev inequality, namely, there exists $S>0$ independent of $j$ and $\epsilon$ such that for all $v\in C^\infty_0(\hat Y_j)$, 
$$\left\|v\right\|^2_{L^{\frac{2n}{n-1}}(\hat Y_j, \hat \omega_{j, \epsilon})}\leqslant S \left\|\nabla_{\hat\omega_{j, \epsilon}}v\right\|^2_{L^2(\hat Y_j, \hat\omega_{j, \epsilon})}.$$

\begin{lemma}\label{C^0 bound}
	There exists a continuous function $\mathcal P: (0, \infty)\rightarrow (0, \infty)$ such that for all $j\in \mathbb Z_{\geqslant 1}$ and $\epsilon\in (0,\lambda^{4j})$,
	
	\begin{itemize}
		\item [(1).]  $\left\Vert\varphi_{j, \epsilon}\right\Vert_{L^{\infty}}\leqslant  \mathcal  P(A_{j, \epsilon}).$
		\item [(2).] $
	 \mathcal  P(A_{j, \epsilon})^{-1}\omega_j\leqslant \hat\omega_{j, \epsilon}\leqslant  \mathcal  P(A_{j, \epsilon})\frac{\mu_j}{\omega_j^n}\omega_j.$
	\end{itemize}
\end{lemma}

  \begin{proof} The proof is standard, see for example \cite[Section 4]{CGT}. For completeness, we include some details.
  
  	(1). We have
  	$\hat\omega^{n}_{j, \epsilon}-\omega_{j, \epsilon}^{n}= \ii\partial\pp \varphi_{j, \epsilon} \wedge \sum_{k=0}^{n-1}\left(\hat\omega^{k}_{j, \epsilon} \wedge \omega_{j, \epsilon}^{n-1-k}\right).$
  	For any exponent $p\geqslant p_0$, multiplying both sides by $|\varphi_{j, \epsilon}|^{p-1}\varphi_{j, \epsilon}$ and integrating by parts (the estimates obtained in Theorem \ref{ existence result of Tian-Yau} justify this), we get 
  	\begin{equation*}
  		\int_{X}|\varphi_{j, \epsilon}|^{p-1}\varphi_{j, \epsilon}(e^{-f_{j, \epsilon}}-1)\mu_j= \frac{4p}{(p+1)^2}\int_{X} \ii\p |\varphi_{j, \epsilon}|^{\frac{p+1}{2}}\wedge \pp |\varphi_{j, \epsilon}|^{\frac{p+1}{2}}\wedge \sum_{k=0}^{n-1}\left(\hat\omega^{k}_{j, \epsilon} \wedge \omega_{j, \epsilon}^{n-1-k}\right).
  	\end{equation*}
  	 In the following $C$ always denote a constant  depending only on $n$ and $p_0$. By Sobolev inequality, we get
    \begin{equation*}
        \left(\int_{X}|\varphi_{j, \epsilon}|^{(p+1)\frac{n}{n-1}}\mu_j\right)^{\frac{n-1}{n}}\leqslant CSp\int_{X}|\varphi_{j, \epsilon}|^{p}|e^{-f_{j, \epsilon}}-1|\mu_j.
    \end{equation*}
Applying the H\"older's inequality to the right hand side, we obtain 
\begin{equation*}
    \begin{aligned}
        \left\Vert\varphi_{j, \epsilon}\right\Vert_{L^{\frac{n(p_0+1)}{n-1}}(\mu_j)}&\leqslant C S \left\Vert e^{-f_{j, \epsilon}}-1\right\Vert_{L^{q_0}(\mu_j)},\\
        \left\Vert\varphi_{j, \epsilon}\right\Vert^{p+1}_{L^{\frac{n(p+1)}{n-1}}(\mu_j)}&\leqslant CSp
		 \left\Vert e^{-f_{j, \epsilon}}-1\right\Vert_{L^{p+1}(\mu_j)}\left\Vert\varphi_{j, \epsilon}\right\Vert^p_{L^{p+1}(\mu_j)} \text{ for all $p\geqslant p_0$}.
    \end{aligned}
\end{equation*}
Then a standard Moser iteration argument completes the proof.

(2).  Note that the second inequality follows from the first one and the elementary inequality that $\operatorname{tr}_{\beta}(\alpha)\leqslant \frac{\alpha^n}{\beta^n}\left(\operatorname{tr}_{\alpha}(\beta)\right)^{n-1}$ for any two positive $(1,1)$-forms $\alpha$ and $\beta$. The first inequality here is a consequence of Chern-Lu's inequality. More precisely, by construction the bisectional curvature of $\omega_j$ has a uniform upper bound independent of $j$. Then Chern-Lu's inequality implies that there is a constant $B$ independent of $j$ and $\epsilon$ such that,
\begin{equation*}
	\Delta_{\hat\omega_{j, \epsilon}}\log(\operatorname{tr}_{\hat\omega_{j, \epsilon}}\omega_j)\geqslant -B \operatorname{tr}_{\hat\omega_{j, \epsilon}}\omega_j.
\end{equation*}
 	By the definition of $\hat\omega_{j, \epsilon}$ and the estimates for $u_{j,\epsilon}$ in Proposition \ref{preconditioning}, we have 
$$
	\operatorname{tr}_{\hat\omega_{j, \epsilon}}(\ii\p\pp \varphi_{j, \epsilon})\leqslant n-\frac{1}{2}\operatorname{tr}_{\hat\omega_{j, \epsilon}}\omega_j.
$$
Therefore,
\begin{equation*}
	\Delta_{\hat\omega_{j, \epsilon}}\left(\log(\operatorname{tr}_{\hat\omega_{j, \epsilon}}\omega_j)-2(B+1)\varphi_{j, \epsilon}\right) \geqslant  \operatorname{tr}_{\hat\omega_{j, \epsilon}}\omega_j-2(B+1)n.
\end{equation*} Now we apply the maximum principle to the function $F=\log(\operatorname{tr}_{\hat\omega_{j, \epsilon}}\omega_j)-2(B+1)\varphi_{j, \epsilon}$. Since $\varphi_{j, \epsilon}$ tends to zero at infinity, $F$ converges to $\log n$ at infinity. Notice $\pi_j^*\omega_j$ is degenerate along the exceptional set, so $F$ tends to $-\infty$ on $\pi_j^{-1}(\text{Sing}(Y_j))$. If the supremum of $F$ is attained at infinity, then we automatically have a uniform upper bound for $\log (\operatorname{tr}_{\hat\omega_{j, \epsilon}}(\omega_j))$. If the maximum is achieved on $\hat Y_j$,  Then $\log (\operatorname{tr}_{\hat\omega_{j, \epsilon}}(\omega_j))\leqslant C$, where $C$ depends only on  $\left\Vert\varphi_{j, \epsilon}\right\Vert_{L^{\infty}}$.
 \end{proof}

  Using this we can also improve Lemma \ref{non uniform estimate}, namely, we can make the dependence of $K_{j,\epsilon}$ and $Q_{k, j, \epsilon}$ on $j$ and $\epsilon$ to be  through  $A_{j, \epsilon}.$ Since we are free to pass to subsequences when taking limits and we do not need uniqueness of the limits, we may assume in the following that $A_j=\lim_{\epsilon\rightarrow0}A_{j,\epsilon}$ for every $j$.
For fixed $j$ since as $\epsilon\rightarrow 0$ we have $A_{j, \epsilon}\rightarrow A_j$, passing to a subsequence $\epsilon_{j,l}$, as $l\rightarrow\infty$ we can make $\varphi_{j, \epsilon_{j, l}}$ converge locally smoothly to a limit $\varphi_j$ and $u_{j, \epsilon_{j, l}}$ converge locally smoothly to a limit $u_j$ on $Y_j^\circ$. The metric $\hat\omega_j=\omega_j+\sqrt{-1}\p\bp(u_j+\varphi_j)$ defines a Calabi-Yau metric on $Y_j\setminus \text{Sing}(Y_j)$, i.e. $\hat{\omega}_j^n=(\ii)^{n^2}\Omega_j\wedge\ols{\Omega}_j.$ Moreover, we have
	\begin{itemize}
		\item On $Y_j^\circ$, \begin{equation}\label{e:phij uniform bound}|\varphi_{j}|\leqslant \mathcal P(A_j).
		\end{equation}This implies $\hat\omega_j$ extends as a closed positive $(1,1)$ current on $Y_j$.
		\item  On $Y_j^{\circ}$, \begin{equation}\label{e:lower bound}\hat\omega_{j}\geqslant \mathcal P(A_j)^{-1}\omega_j.
		\end{equation}
		\item  There exist compact sets $K_j$ depending only on $A_j$ and continuous functions $Q_k:(0, \infty)\rightarrow (0, \infty)$ for all $k\geqslant0$ such that on $Y_j\setminus K_j$ $$|\nabla^k_{\omega_{j}}\varphi_{j}|\leqslant  Q_k(A_j) r_j^{-\delta_3-k}.$$  
  \item there exists a compact set $K\subset \mathbb C^N$ such that $u_j$ is a constant on $K\cap Y_j$ and on $Y_j^\circ$ $$\left|\nabla^k_{\omega_j}u_{j}\right|_{\omega_j}\leq C_k (r_j+1)^{-\delta_2+2-k}.$$
	\end{itemize}
	Moreover, a standard argument using the Bishop-Gromov volume comparison (see  \cite[Lemma 4.14]{CGT}) implies that for any compact set $K\subset\C^N$, there is a constant $C_K>0$ independent of $j$ and $\epsilon$ such that \begin{equation*}
		\diam (\pi_j^{-1}(K\cap Y_j),\hat\omega_{j, \epsilon})\leqslant C_K.
	\end{equation*}
	So the diameter of $K\cap Y_j^\circ$ with respect to $\hat\omega_j$ is also uniformly bounded independent of $j$.

By construction  we know that as $j\rightarrow\infty$, $\omega_j$ converges smoothly to $\omega_\infty=\sqrt{-1}\p\bp \psi_\infty$ on $W\setminus \{0\}$, and $\psi_\infty$ is continuous across $0$.   Using the above estimates and Lemma \ref{l:Aj uniform bound} by passing to a subsequence we may assume that $\varphi_j$ converges locally smoothly to a smooth function $\varphi_\infty$ and $u_j$ converges locally smoothly to a smooth function $u_{\infty}$ and on $W\setminus \{0\}$, such that $\hat\omega_\infty=\omega_\infty+\sqrt{-1}\p\bp(u_{\infty}+\varphi_\infty)$ is a K\"ahler Ricci flat metric on $W\setminus\{0\}$.Moreover $u_{\infty}$ is a constant near $0$, $\varphi_{\infty}$ is globally bounded,
$\hat\omega_\infty\geqslant C^{-1}\omega_\infty$ and there exists a compact set $K$ such that on $W\setminus K$
\begin{equation}\label{estimate for u and phi}
|\nabla_{\omega_\infty}^ku_\infty|\leqslant C_k r_\infty^{-\delta_2+2-k},\quad|\nabla_{\omega_\infty}^k\varphi_\infty|\leqslant C_k r_\infty^{-\delta_3-k},
\end{equation}
  where we recall $r_\infty=r\circ p_\infty^{-1}$. Notice that since $\Omega_j$ is parallel with respect to $\hat \omega_j$, we may assume $\Omega_j$ converges to a holomoprhic volume form $\Omega_{\infty}$ on $W\setminus\{0\}$ such that $\hat{\omega}^n_\infty=(\ii)^{n^2}\Omega_{\infty}\wedge\ols{\Omega}_{\infty}.$  Moreover the metric $\hat \omega_{\infty}$ has finite diameter on $K\cap W$, and it is asymptotic to the cone $\mathcal C$ at infinity.

\section{Proof of the main results}\label{s6}
We now investigate the singular behavior of  $\hat\omega_\infty$ at $0$. To this end we need to study the singular behavior of $\hat\omega_j$, so we first fix $j$ and consider the family $\hat\omega_{j, l}\equiv \hat\omega_{j, \epsilon_{j,l}}$ as $l\rightarrow\infty$. 
Passing to a further subsequence we let $({Z_j},d_{{Z_j}})$ be a pointed Gromov-Hausdorff limit of $(\hat Y_j, \hat\omega_{j, l}, p_j)$ (where $p_j$ is a point on the exceptional set of $\pi_j)$ as $l\rightarrow\infty$. As mentioned before, the metrics $\hat\omega_{j, \epsilon}$ are uniformly volume non-collapsing, so the general convergence theory leads to a decomposition  ${Z_j}=\mathcal R\cup \mathcal S$,  where the regular set $\mathcal R$ is a smooth open  manifold with $d_{Z_j}|_{\mathcal R}$ induced by a Calabi-Yau metric $(\omega_{\mathcal R}, J_{\mathcal R})$, and the singular set $\mathcal S$ is closed and has Hausdorff codimension at least 4.  
 
 \begin{proposition}\label{p: metric completion}
${Z_j}$ agrees with the metric completion of $(Y_j^\circ, \hat\omega_j)$ and is naturally homeomorphic to $Y_j$. 
\end{proposition}
 
This would likely follow from the arguments of J. Song \cite{Song} (and its generalization in \cite{CGT}). We outline a different and more direct proof here, which may be of independent interest. It is not hard to see that we can naturally view $ (Y_j^\circ,\hat\omega_j)$ as embedded in $(\mathcal R, \omega_{\mathcal R})$ as an open K\"ahler manifold. Moreover, since $\omega_j\leqslant \mathcal{P}(A_j)\hat\omega_{j,l}$, by passing to a further subsequence, one can take the limit of the holomorphic maps $\pi_j: (\hat Y_j, \hat\omega_{j, l})\rightarrow Y_j\subset \C^N$ and obtain a surjective continuous map $\pi_\infty: {Z_j}\rightarrow Y_j\subset \C^N$, which is the identity on $Y_j^\circ$ and holomorphic on $\mathcal{R}$. It then follows that $\mathcal R\setminus Y_j^\circ$ is a complex analytic set in $\mathcal R$, given by $\pi_\infty^{-1}(\text{Sing}(Y_j))\cap \mathcal{R}$.
By the discussion in Section \ref{s5}, on $Y_j^\circ$ we may write $\hat\omega_j=\sqrt{-1}\p\bp\hat\varphi_j$, where $\hat\varphi_j$ is locally bounded across $\text{Sing}(Y_j)$. A key point is the following 

\begin{lemma}\label{l:potential gradient bound}
The function $\hat\varphi_j$ extends smoothly to $\mathcal R$. Moreover, for any compact set $K\subset {Z_j}$ we have
    $$\sup_{K\cap \mathcal R}\left|\nabla_{\omega_{\mathcal R}}\hat\varphi_j\right|_{\omega_{\mathcal R}}<\infty.$$
\end{lemma}
\begin{proof}
On $Y_j^{\circ}$, the form   $\sqrt{-1}\p\bp\hat\varphi_j=\hat\omega_j=\omega_{\mathcal R}$ is smooth. Since $\mathcal R\setminus Y_j^{\circ}$ is a closed pluripolar set and $\hat\varphi_j$ is locally bounded across $\text{Sing}(Y_j)$, by well-known removable singularity theorem for plurisubharmonic functions (see \cite[Chapter I,Theorem 5.24]{Demailly}), we obtain that the same equation holds on $\mathcal R$ in the sense of currents, so elliptic regularity gives a smooth extension of $\hat\varphi_j$ to $\mathcal R$.

To prove the gradient bound we fix an open subset $U\subset Y_j$ containing $\text{Sing}(Y_j)$ and with $\p U$ smooth. Denote  $\hat U=\pi_j^{-1}(U)$. Its boundary $\p\hat U$ is naturally identified with $\p U$. We can solve the Poisson equation $\Delta_{\hat\omega_{j, l}}\eta_{l}=n$ on $\hat U$ with boundary condition given by $\hat\varphi_j|_{\p U}$. Then by the standard Moser iteration and Cheng-Yau gradient estimate we have $|\nabla_{\hat\omega_{j, l}} \eta_l|\leqslant C$ for all $l$. So passing to a subsequence $\eta_l$ converges uniformly to a limit $\eta$ on  $\pi_\infty^{-1}(U)\subset {Z_j}$, with  $|\nabla_{\hat\omega_j}\eta|\leqslant C$,   $\Delta_{\hat\omega_j}\eta=n$ on $\pi_\infty^{-1}(U)\cap \mathcal R$ and $\eta|_{\p U}=\hat\varphi_j|_{\p U}$.  

Now consider the function $f=\hat\varphi_j-\eta$ on $\pi_\infty^{-1}(U)\cap\mathcal R$. It is harmonic and $f|_{\p U}=0$. One can conclude that $f$ is identically zero from the standard removable singularity theorem for harmonic functions on Ricci limit spaces. For example, using the fact that the Hausdorff codimension of $\mathcal S$ is at least 4, for $\epsilon>0$ small, one can construct (see \cite[Proposition 3.5]{DS1}) a cut-off function $\chi_\epsilon: {Z_j}\rightarrow \mathbb R$ which is equal to 1  when $d_{Z_j}(x, \mathcal S)\geqslant \epsilon$ and vanishes in a neighborhood of $\mathcal S$, such that $\int_{Z_j} d_{Z_j}(\cdot, \mathcal S)^{-1} |\nabla_{\omega_\mathcal R} \chi_\epsilon|\omega_{\mathcal R}^n<\epsilon$. By the local Cheng-Yau gradient estimate we have $|\nabla_{\omega_{\mathcal R}}\hat\varphi_j|\leqslant Cd_{Z_j}(\cdot, \mathcal S)^{-1}$. Then by a straightforward integration by parts argument one sees that $f\equiv0.$
\end{proof}

Given this Lemma, one can follow the strategy of \cite{DS1, DS2} to construct holomorphic functions by the H\"ormander $L^2$ technique (see a related observation in \cite{Gabor2} for studying the asymptotic cone of complete $\sqrt{-1}\p\bp$-exact Calabi-Yau manifolds with Euclidean volume growth).  Here one can work on the trivial holomorphic line bundle $L$ over $Y_j^{\circ}$ endowed with the hermitian metric $|\cdot|_{\hat\varphi_j}\equiv e^{-\hat\varphi_j}$. Notice $Y_j^{\circ}$ admits a complete K\"ahler metric, so one can apply Theorem \ref{Hormander} here. Let $\nabla_L$ be the corresponding Chern connection on $L$. For $k\geqslant 1$, denote by $\mathcal V_k$ the space of holomorphic sections $f$ of $L^k$ over $Y_j^{\circ}$ with $$\|f\|_k^2\equiv \int_{Y_j^{\circ}}|f|^2_{k\hat\varphi_j}(k\hat{\omega}_j)^n<\infty.$$
Since $Y_j$ is normal it is clear that any holomorphic function $f$ on $Y_j^{\circ}$ extends to a holomorphic function on $Y_j$. In particular, for any compact set $K\subset \C^N$,  $|f|$ and $|df|_{\omega_j}$ are uniformly bounded over  $Y_j^{\circ}\cap K$. By \eqref{e:phij uniform bound}, \eqref{e:lower bound} and Lemma \ref{l:potential gradient bound}, we know that $|f|_{k\hat\varphi_j}$ and $|\nabla_L f|_{k\hat\varphi_j}$ are also uniformly bounded over  $Y_j^{\circ}\cap K$. Viewed in ${Z_j}$, it follows that $f$ extends to a holomorphic function on $\mathcal R$ and the above norms are locally bounded near $\mathcal S$.

To apply the idea in \cite{DS1} we need to show that for any compact set $K\subset Y_j$, there exist constants $K_0, K_1>0$ depending on $K$ such that for all $k\geqslant 1$ and any $f\in \mathcal V_k$, we have 
\begin{equation}
    \sup_K|f|_{k\hat\varphi_j}\leqslant K_0 \|f\|_k
,\ \ \ \ \sup_K|\nabla_L f|_{k\hat\varphi_j}\leqslant K_1\|f\|_k.
 \end{equation}
These were proved by applying the Moser iteration to the corresponding differential inequalities  $$\Delta_{k\hat\omega_j}|f|_{k\hat\varphi_j}\geqslant -C|f|_{k\hat\varphi_j}, \ \ \ \ \Delta_{k\hat\omega_j}|\nabla_L f|_{k\hat\varphi_j}\geqslant -C|\nabla_L f|_{k\hat\varphi_j}, $$
which a priori only hold on $\mathcal R$, but again using the existence of a good cut-off function as in \cite[Proposition 3.5]{DS1},  and noticing that there is a Sobolev inequality on ${Z_j}$ (since $Z_j$ is a volume non-collapsing Ricci limit space), one can make the usual arguments go through.
 
Using the existence of  tangent cones at points in ${Z_j}$, it is straightforward to adapt the arguments in \cite{DS1, DS2} to construct holomorphic functions in $\mathcal V_k$ that separate points. Namely, given any $x_1\neq x_2\in {Z_j}$, one can find $k\geqslant 1$ and $f_1, f_2\in \mathcal V_k$ such that 
$  |f_\alpha(x_\alpha)|_{k\hat\varphi_j}\geqslant 1,  $ and  $|f_\alpha(x_\beta)|_{k\hat\varphi_j}\leqslant 1/10$ for $\alpha\neq\beta$.

\begin{proof}[Proof of Proposition \ref{p: metric completion}]
We first prove $\pi_\infty$ is injective (hence is a homeomorphism). Suppose  $x_1\neq x_2\in {Z_j}$. Then we can construct the functions $f_1, f_2$ as above. But both can be viewed as holomorphic functions on $Y_j$, so it follows that $\pi_\infty(x_1)\neq \pi_\infty(x_2)$. 

Since $\text{Sing}(Y_j)$ consists of finitely many points,  we see ${Z_j}\setminus Y_j^{\circ}=\pi_\infty^{-1}(\text{Sing}(Y_j))$ is also a finite set. Clearly this implies that ${Z_j}$ is indeed the metric completion of $(Y_j^{\circ}, \hat\omega_j)$. 
\end{proof}

 Now we further let $j\rightarrow\infty$. Passing to a subsequence we may take a pointed Gromov-Hausdorff limit $(Z_j, d_{Z_j})\rightarrow (Z_\infty, d_\infty)$. The same argument as above combined with the uniform estimates in Section \ref{s5} gives  
 
 \begin{proposition}
 $Z_\infty$ agrees with the metric completion of $(W\setminus \{0\}, \hat\omega_\infty)$ and is naturally homeomorphic to $W$.
 \end{proposition}
In particular,  we may naturally identify $Z_\infty$ with $W$. 

\begin{proof}[Proof of  Theorem \ref{t:main0}] As is proved in Section \ref{s5}, $Z_\infty$ has a unique asymptotic cone given by $\mathcal C$.
From the above discussion it is straightforward to extend the arguments of \cite{DS2}  to show that there is a unique metric tangent cone $\mathcal C'$ at $0$.  Furthermore, by the minimizing property of K-semistable valuations proved by Li-Xu \cite{LX}, we know the volume densities satisfy $\kappa(\mathcal C')\leqslant \kappa(\mathcal C)$, where the right hand side can be interpreted as the normalized volume associated to the natural valuation of $W$ at $0$.
Notice $Z_\infty$ is a pointed Gromov-Hausdorff limit of complete Ricci-flat manifolds, so the Bishop-Gromov inequality applies to $Z_\infty$, which gives that $\kappa(\mathcal C')\geqslant \kappa(\mathcal C)$. Hence the equality holds and then  $Z_\infty$ must be a K\"ahler cone itself, hence is isomorphic to $\mathcal C$ as a K\"ahler cone.  Denote by $\xi'$ the Reeb vector field on $Z_\infty$ and by $\tilde r$ the radial function on $Z_\infty$.

It remains to show that $\xi'=\xi$. For this  we need to use some results in \cite{DS2}. The algebraic structure of $Z_\infty$ can be intrinsically characterized in terms of the ring $R(Z_\infty)$ holomorphic functions on $Z_\infty$ of polynomial growth (measured using the K\"ahler cone metric). The cone structure yields a decomposition \begin{equation}\label{e:coordinate ring decomposition}R(Z_\infty)=\bigoplus_{\mu\geqslant0} R_\mu(Z_\infty),
\end{equation}where $R_\mu(Z_\infty)$ consists of $f$ which are of homogeneous of degree $\mu$, i.e., $\mathcal L_{J\xi'}f=-\mu f$.  Conversely, $\xi'$ can be recovered from this decomposition.

The above decomposition is determined by a filtration associated with a \emph{degree} function. 
For any nonzero $f\in R(Z_\infty)$, the following is well-defined (see \cite{DS2})
$$d_{Z_\infty}(f)\equiv\lim_{R\rightarrow\infty}\frac{\sup_{\tilde r(x)=R}\log |f(x)|}{\log R}<\infty.$$
Then the decomposition in \eqref{e:coordinate ring decomposition} is determined by the filtration of $R(W)$ with respect to $d_{Z_\infty}$.

 On the other hand,   the coordinate ring $R(W)$ of $W$ consists of restrictions of polynomial functions on $\C^N$. We define
 $$d_{W}(f)\equiv\lim_{R\rightarrow\infty}\frac{\sup_{r_\xi(x)=R}\log |f(x)|}{\log R}.$$
 There is  a decomposition $R(W)=\bigoplus_{\mu\geqslant 0} R_\mu(W)$, where $R_\mu(W)$ is the space of polynomials $f$ with $\mathcal L_{J\xi}f=-\mu f$. It can be checked that this decomposition is also determined by the filtration of $R(W)$
 with respect to $d_W$.
 
Notice that by the construction of $\hat\omega_\infty$  we know for all $\epsilon>0$ ,$C_\epsilon^{-1}r_\xi^{1-\epsilon}\leqslant \tilde r\leqslant C_\epsilon r_\xi^{1+\epsilon}$ on $W\setminus K$ for a fixed compact set $K$. It follows that $R(Z_\infty)=R(W)$ and $d_W(f)=d_{Z_\infty}(f)$ for any $f\in R(W)$. So we must have $R_\mu(Z_\infty)=R_\mu(W)$ for all $\mu\geqslant 0$. 
 In particular, $\xi'=\xi$. Hence we have constructed a $T_\xi$ equivariant isomorphism between $W$ and $\mathcal C$. This completes the proof of Theorem \ref{t:main0}. 

\end{proof}

\begin{proof}[Proof of Theorem \ref{t:main2}] We simply notice that by the estimate of $\varphi_\infty$ and $u_{\infty}$ in \eqref{estimate for u and phi}, the background metric $\omega_\infty$ is polynomially asymptotic to the Calabi-Yau cone metric $\hat \omega_{\infty}$ on $W$  (which is the same as $\mathcal C$). Moreover by the construction of $\omega_{\infty}$ (see Corollary \ref{c:metric close} and Lemma \ref{construction of exact background kahler forms}), it is polynomially asymptotic to $\omega$ under the diffeomorphism $p_{\infty}$ . Then $p_\infty$ provides the diffeomorphism that shows $\omega$ and $J$ are asymptotic to $\hat\omega_\infty$ and $J_W$ in a polynomial rate. The asymptotics of $\Omega$ follows from \eqref{fix angle} by letting $j\rightarrow \infty$.

\end{proof}

\begin{remark}Using the same proof one can weaken the Calabi-Yau condition in Theorem \ref{t:main0} and \ref{t:main2} to be Ricci-flat K\"ahler. Alternatively, one can show this by working on the universal cover since it is well known \cite{Anderson, PLi} that Riemannian manifolds with non-negative Ricci curvature and Euclidean volume growth have finite fundamental groups. 
\end{remark}

\section{Discussions}

We point out some further related directions that  one can explore.  The guiding problem is

\begin{problem}[Algebraization for complete Calabi-Yau metrics]\label{q6.1}
	Give an algebro-geometric characterization of all complete Calabi-Yau metrics with Euclidean volume growth. 
\end{problem}
\begin{remark}
Without the volume growth condition, the situation is more complicated and extra assumptions are needed in order to make connections with algebraic geometry. In complex dimension 2 under a natural finite energy condition, the models at infinity are completely classified by Sun-Zhang \cite{SZ}, and there is a complete classification of the Calabi-Yau metrics in terms of algebro-geometric data by the work of many people; see \cite{SZ} for references. 
\end{remark}

Given Theorem \ref{t:main2}, the results of Conlon-Hein \cite{CH3} provide an answer to Problem \ref{q6.1} under the extra quadractic curvature decay  condition. Namely, such Calabi-Yau metrics are always constructed as follows (ignoring uniqueness issues for the moment), where each step is essentially algebro-geometric. 
\begin{enumerate}
	\item Choose a Calabi-Yau cone $(\mathcal C, \xi)$. This is equivalent to choosing a K-polystable Fano cone \cite{CS}.
	\item Choose a nomal affine variety $X'$ with $(\mathcal C, \xi)$ as a weighted asymptotic cone at infinity.
	\item Choose a crepant K\"ahler resolution $\pi:X\rightarrow X'$. 
	\item Choose a K\"ahler class on $X$. This is a numerical condition according to \cite{CoTo,CH3}.  
\end{enumerate}

We want to ask how much of the above picture holds in general, if we allow  singularities in the relevant objects. There are several technical obstacles to realize this program, mostly related to the possible appearance of singularities on the asymptotic cone.
First one needs to extend the 2-step degeneration theory of \cite{DS2} to this setting. In particular, we would like to construct abundant holomorphic functions with polynomial growth on $X$. This is related to a special case of Yau's compactification conjecture, which we reiterate as follows
\begin{conjecture}[Yau's Compactification conjecture] \label{Yau conjecture}
    A complete Calabi-Yau manifold with Euclidean volume growth is naturally a quasi-projective variety.
\end{conjecture}

There has been partial progress towards this. For example, it is known that any asymptotic cone is naturally a normal affine variety \cite{LS}. It is also observed in \cite{Gabor2} that the results of \cite{DS2} automatically extend to the setting when the Calabi-Yau metric is $\p\bp$-exact  so Conjecture \ref{Yau conjecture} holds in this case. Now suppose the 2-step degeneration theory works. The following is natural to expect
 
\begin{conjecture}[No semistability at infinity, the general version]\label{conj6.2}
	Theorem \ref{t:main0} holds without the  quadratic curvature decay assumption. 
\end{conjecture}
\begin{remark}
One can also formulate a conjectural general version of Theorem \ref{t:main2}, where one uses the Gromov-Hausdorff distance to measure the rate of convergence to the asymptotic cone. 
\end{remark}

One difficulty in proving Conjecture \ref{conj6.2} using the strategy of this paper is related  to the generalization of the Tian-Yau construction, which is itself an interesting question. A first step would be 

\begin{question}[Prescribing asymptotic cone]\label{q6.3}
	Given a singular Calabi-Yau cone $(\mathcal C, \xi)$ and a normal affine variety $X'$ with $(\mathcal C, \xi)$ as a weighted asymptotic cone at infinity, when does $X'$ admit a (possibly singular) Calabi-Yau metric asymptotic to $(\mathcal C, \xi)$ in the Gromov-Hausdorff sense?
\end{question}

There are some recent extensions of the Tian-Yau construction with singular asymptotic cones for special examples; see \cite{HN, YLi, Gabor1, CR, Chiu}. 
\

As mentioned in Section \ref{two step degeneration} the classification of complete Calabi-Yau metrics with Euclidean volume growth on a fixed underlying algebraic variety is a subtle problem. Even on $\C^n(n\geqslant 3)$ the situation can be complicated; the recent works by Sz\'ekelyhidi \cite{Gabor2} and Chiu \cite{Chiu} make progress in this direction.  We make an attempt here to formulate some general questions.  To minimize technical issues  we restrict to the case of $\p\bp$-exact metrics on smooth affine varieties. As mentioned above the results of \cite{DS2} apply here. Fix a smooth affine variety $X$ with the coordinate ring $R(X)$.  A complete $\p\bp$-exact Calabi-Yau metric $\omega$ on $X$ with Euclidean volume growth is said to be a \emph{compatible} metric if the space of holomorphic functions on $X$ with polynomial growth with respect to $\omega$ coincides with $R(X)$.
A compatible Calabi-Yau metric $\omega$ defines a degree function $d_\omega:R(X)\rightarrow \mathbb R_{\geq0}\cup \{\infty\}$ satisfying
\begin{itemize}
\item $d_\omega(f)=\infty$ if and only if $f=0$;
 \item $d_\omega(f)=0$ if and only if $f$ is a nonzero constant;
    \item $d_\omega(fg)=d_\omega(f)+d_\omega(g)$;
    \item $d_\omega(f+g)\leqslant \max(d_\omega(f), d_\omega(g))$. 
\end{itemize}
Furthermore, $d_\omega$  gives rise to a filtration of $R(X)$ whose associated graded ring is the coordinate ring $R(W)$ of the intermediate K-semistable cone $W$, and $W$ degenerates to the unique asymptotic cone $\mathcal C$ as affine algebraic cones. The result of Li-Wang-Xu \cite{LWX} implies that $(\mathcal C, \xi)$ is uniquely determined by $d_\omega$.
Clearly, if $\omega$ is compatible, then for any $\lambda>0$ and any algebraic automorphism $F$ of $X$, $\lambda F^*\omega$ is also compatible. Furthermore,  $d_\omega=d_{\lambda\omega}$ and  $d_{F^*\omega}=F^*d_{\omega}$.

In terms of algebro-geometric language one notices that $v_\omega\equiv-d_\omega$ is a valuation on $R(X)$ with values in $\mathbb R_{\leqslant0}\cup\{-\infty\}$. We call such a valuation a \emph{negative valuation}. We say a negative valuation $v$ is \emph{semistable} (resp. \emph{polystable}) if it defines a filtration whose associated graded ring is finitely generated and defines a K-semistable  (resp. K-polystable) Fano cone $\mathcal C_v$ (in the sense of \cite{CS}). Geometrically, the cone is a weighted asymptotic cone at infinity of $X$ under some affine embedding. Conjecture \ref{conj6.2} implies that the negative valuation induced by a compatible Calabi-Yau metric on $X$ is always polystable. 

Given a semistable valuation $v$ on $X$, we define $G_v$ as the group of algebraic automorphisms of $X$ that preserve $v$ and define $\tilde G_v$ as the group of algebraic automorphisms of $\mathcal C_v$ that fixes the vertex and preserves the cone structure. Both groups are finite-dimensional. Notice that $\tilde G_v$ contains the algebraic torus $\mathbb T_v$ generated by the cone vector field. There is a natural homomorphism $\varphi: G_v\rightarrow \tilde G_v$. 

\
\

To classify the space of all compatible Calabi-Yau metrics on $X$, assuming Conjecture \ref{conj6.2}, one can proceed as follows:

\

\textbf{Step 1}. Determine the space $\mathcal V$ of all possible polystable negative valuations on $X$ up to algebraic automorphism.  This is essentially an algebro-geometric question. One can also study the related question with polystability replaced by semistability. In general there may be infinitely many polystable  negative valuations on a given $X$. Indeed,  given any polystable cone $V$ in $\C^n$ which is a hypersurface,  one can find a polystable  negative valuation $v$ on $X=\C^n$ with $\mathcal C_v=\C\times V$. Indeed,  the result of Sz\'ekelyhidi \cite{Gabor1} implies that if $V$ has only isolated singularities, then the associated negative valuation actually arises  from a compatible Calabi-Yau metric on $\C^n$. Still, it seems interesting to investigate the structure of the space $\mathcal V$.

\

\textbf{Step 2}. Now we fix  a polystable negative valuation $v$ on $X$. Denote by $\mathcal M_v$ the space of all compatible Calabi-Yau metrics $\omega$ on $X$ with $v=v_\omega$.  Is $\mathcal M_v$ nonempty?
This is related to Question \ref{q6.3}.

\

\textbf{Step 3}. Suppose $\mathcal M_v\neq \emptyset$. We know that all $\omega\in \mathcal M_v$ have the same asymptotic cone at infinity given by $\mathcal C_v$.  But we should expect more and one would like to have

\begin{conjecture}[Uniform equivalence]\label{conj6.5}
    There exists a constant $C>0$ such that for all $\omega_1, \omega_2\in \mathcal M_v$, we have $C^{-1}\omega_1\leqslant \omega_2\leqslant C\omega_1$. 
\end{conjecture}
It is easy to see the converse is true:  if two compatible Calabi-Yau metrics $\omega_1$ and $\omega_2$ are uniformly equivalent, then $v_{\omega_1}=v_{\omega_2}$. We remark that one can also ask an analogous question for local singularities of K\"ahler-Einstein metrics (on algebraic varieties with klt singularities), that is, whether any two such metrics are indeed locally uniformly equivalent. The answer to this question, currently not known,  is also related to Conjecture \ref{conj6.5}. The recent work by Chiu-Sz\'ekelyhidi \cite{ChSz} made progress in understanding local singular behavior of polarized K\"ahler-Einstein metrics.

\

\textbf{Step 4}. Denote by $\mathcal N_v$ the space of Calabi-Yau cone metrics on $\mathcal C_v$ (with the given Reeb vector field determined by $v$). Notice $\mathcal C_v$ can be realized as a weighted asymptotic cone at infinity for some embedding $X\subset \C^N$. One is attempted to define a natural  map $\mathfrak C: \mathcal M_v\rightarrow \mathcal N_v$ given by taking the (appropriate) rescaled limit of the K\"ahler form under the weighted asymptotic cone construction.

\begin{conjecture}[Generalized uniqueness]\label{conj6.7}
    The map $\mathfrak C$ is well-defined and is bijective. 
\end{conjecture}

We explain Conjecture \ref{conj6.5} and \ref{conj6.7} indeed hold if we assume $\mathcal C_v$ has only an isolated singularity at the vertex. In this case any Calabi-Yau metric in $\mathcal M_v$ has quadratic curvature decay so we may apply the arguments of this paper. Since we have fixed the negative valuation $v$, this fixes the algebro-geometric degeneration from $X$ to $\mathcal C_v$, in terms of a weighted asymptotic cone. Then by Section \ref{section:defining projection map} we can fix the diffeomorphism between the end of $X$ and the end of $\mathcal C_v$. For every $\omega_{\infty}\in \mathcal N_v$, by \cite{CH1} (or the arguments in Section \ref{s5} of this paper) one gets a compatible Calabi-Yau metric on $X$ which is asymptotic to $\omega_\infty$ at a polynomial rate. Conversely, given a compatible Calabi-Yau metric $\omega\in \mathcal M_v$, using the arguments in this paper, we can get a Calabi-Yau cone metric on $\mathcal C_v$ which is the asymptotic cone of $\omega$ at a polynomial rate. One can see this is exactly the $\mathfrak C$ map that we want. Furthermore, it is independent of the auxiliary data chosen. The uniqueness theorem \cite[Theorem 3.1]{CH1} proved by Conlon-Hein implies that $\mathfrak C$ is bijective. 

Notice even in this case it is not clear that the action of $G_v$ is transitive on $\mathcal M_v$ (modulo scaling): we know that the action of $\widetilde G_v$ on $\mathcal{N}_v$ is transitive by the generalized Bando-Mabuchi theorem (see for example, \cite[Proposition 4.8] {DS2}). Indeed, $\mathcal N_v$ can be identified with the homogeneous space $\widetilde G_v/K_v$ for a maximal compact subgroup $K_v$ whose complexification is $\widetilde G_v$.  But the  map $\varphi: G_v\rightarrow\tilde G_v$ could be complicated in general: neither the injectivity nor the surjectivity seems obvious. It seems to us that Conjecture \ref{conj6.7} is a more natural formulation than a naive uniqueness statement. We mention that the above conjecture is compatible with the result and conjecture in \cite{Gabor2, Chiu} for complete Calabi-Yau metrics on $\C^3$ with specific asymptotic cones.

\

The above four steps altogether would lead to a satisfactory classification of compatible Calabi-Yau metrics on  a fixed smooth affine variety. More generally, when the metrics are not necessarily $\p\bp$-exact, one  needs to further involve K\"ahler classes in the discussion. 

\

From another perspective, one can also try to classify complete Calabi-Yau metrics asymptotic to a given Calabi-Yau cone, without fixing the underlying algebraic variety.  
Notice again this should contrast with the local situation. As is well known (see \cite{DS2, HN, Gabor1}), there is no boundedness of local singularities of Calabi-Yau metrics with a given tangent cone if the latter has nonisolated singularities. One example is given by the $A_k$ singularities in dimension $n\geqslant 3$ for $k\geqslant 4$, whose local tangent cones are all given by the  product of an $n-1$ dimensional $A_1$ singularity with $\mathbb C$. However, for asymptotic cones, there are more constraints. For example, the Bishop-Gromov inequality ensures a strong uniform non-collapsing condition, which gives a preliminary compactness in this situation. In particular, one would expect a well-behaved moduli space when the volume density of the asymptotic cone is bounded away from zero.

\begin{problem}[Moduli problem]
Given a positive number $\lambda>0$, study the structure of the moduli space of  (possibly singular) $\p\bp$-exact Calabi-Yau metrics on affine varieties whose asymptotic cones have volume density at least $\lambda$.

\end{problem}

 It is straightforward to formulate an algebro-geometric counterpart of the above problem. That is,  one can study the moduli space $\mathfrak M_{\geqslant \lambda}$ of pairs $(X, v)$, where $X$ is an affine variety and $v$ is a polystable negative valuation on $X$ with $\mathcal C_v$ having volume density at least $\lambda$. It seems to be an approachable question given recent advances in birational algebraic geometry. However, as pointed out by Yuji Odaka there are some complications. For simplicity we fix a K-polystable cone $\mathcal C$, then the deformation space $\operatorname{Def}(\mathcal C)$ of $\mathcal C$ as affine varieties is in general infinite dimensional, but it is natural to expect that the subspace $\operatorname{Def}_{-}(\mathcal C)$  which consists of those deformations to the above pairs $(X, v)$ with $\mathcal C_v=\mathcal C$ is  finite dimensional, and there is an action of $\widetilde G_v$ on $\text{Def}_{-}(\mathcal C)$. Now since $\widetilde G_v$ is noncompact the quotient is in general not Hausdorff, so any meaningful construction of such moduli space will have to address this issue. On the other hand, if we accept the above picture, then the space of $\p\bp$-exact Calabi-Yau metrics asymptotic to $\mathcal C$ (up to isomorphism) is given by the quotient $\text{Def}_{-}(\mathcal C)/K_v$. It seems an interesting question to us to explore the structure here in more detail.

\

Finally, we point out that while Theorem \ref{t:main2} is a Riemannian geometric statement, our proof in this paper hinges on complex geometry.
One can  ask whether the polynomial convergence rate
is a more general  phenomenon in geometric analysis, or is it  only special in the complex geometric world. Notice this is a question involving \emph{global} geometry, not only the end at infinity. Indeed, we expect that the technique of Adams-Simon \cite{AS} can be used to show that given a Calabi-Yau cone $\mathcal C$ with obstructed deformations in an appropriate sense, one can produce an end of Calabi-Yau metric which is asymptotic to $\mathcal C$ at only a logarithmic rate; see \cite{GChen} for a related work on isolated singularities of $G_2$ metrics.  Such an end is also interesting in terms of complex geometry since it can \emph{not} be filled in as a complete K\"ahlerian manifold -- otherwise using the Tian-Yau construction one would get a contradiction with Theorem \ref{t:main2}.

\begin{question}[Polynomial rate to the asymptotic cone]\label{conj6.9}
    Is a complete Ricci-flat Riemannian manifold with Euclidean volume growth always asymptotic to a unique metric cone at a polynomial rate (in the Gromov-Hausdorff distance)?
\end{question}

Even the uniqueness of asymptotic cones is not known to date without the quadratic curvature decay condition. An answer to Question \ref{conj6.9}, either in the positive or the negative, would  be interesting.   A special setting is given by manifolds with special holonomy, i.e., $G_2$ and $Spin(7)$ manifolds. One can also explore similar questions for global solutions of other geometric PDEs, such as minimal submanifolds and Yang-Mills connections. These will be important in classifying singularity models.

\printbibliography

\end{document}